\documentclass{article}
\input{glyphtounicode}
\usepackage{amsmath,amssymb,amsthm}
\usepackage{array}
\usepackage{booktabs}
\usepackage[font=small,labelfont=bf]{caption}
\usepackage{enumitem}
\usepackage{xcolor}
\usepackage{tikz}
\usetikzlibrary{decorations.pathreplacing,arrows.meta}
\usepackage{microtype}
\usepackage{url}
\usepackage[hidelinks]{hyperref}

\setlist{topsep=4pt,itemsep=2pt}
\setlist[itemize]{label=$\bullet$}

\definecolor{cblue}{RGB}{0,60,230}
\definecolor{cred}{RGB}{220,0,0}
\definecolor{cgreen}{RGB}{0,150,0}
\definecolor{cmag}{RGB}{200,0,200}
\definecolor{corange}{RGB}{230,120,0}
\newcommand{\hl}[2]{\textcolor{#1}{\mathbf{#2}}}

\newtheorem{theorem}{Theorem}
\newtheorem{lemma}{Lemma}
\theoremstyle{definition}
\newtheorem*{definition}{Definition}
\newtheorem*{example}{Example}

\newcommand{\op}[1]{\mathrm{#1}}

\title{Making Every Number from 1 to $N$\\ Under a Fixed Cycle of $+$, $\times$, $-$, $\div$}
\author{Sean Lesmana \and Theodore Tjugiarto}
\date{September 2026}

\begin{document}

\maketitle

\begin{abstract}
Start with the number 2. At each move, combine two numbers already made, but the operations must be used in the fixed repeating order $+, \times, -, \div$. We ask for the fewest moves needed to make every integer from 1 to $N$. Since 2 is already one of the numbers we want and each move makes at most one new number, at least $N-1$ moves are needed. We show that $N-1$ moves are also enough for every $N \ge 9$. The usual induction from $n$ to $n+1$ cannot work, because once all of $1, \dots, P$ have been made, a division can only produce a number already made. Instead, once $1, \dots, P$ are complete, we make all of $P+1, \dots, 3P$ as one block; counting shows that 3 is the smallest factor for which a block extension of this kind can work. We then adjust the last few moves to reach every other $N$. For $9 \le N \le 33$ we give explicit sequences, found by computer search.
\end{abstract}

\section{Introduction}

Start with the number 2. At each move we combine two numbers we already have to make a new one. Numbers can be reused, so $2+2=4$ and then $2\times4=8$ is allowed. The operations, however, must follow the fixed order
\[
  +,\ \times,\ -,\ \div,\ +,\ \times,\ -,\ \div,\ \dots
\]
and a division is only allowed when the answer is a positive integer.

How few moves does it take to make all of $1, 2, \dots, N$? The challenge with the fixed order is that the cycle sometimes calls for a division when we would rather add, and a division need not even yield a whole number. We show that for $N \ge 9$ the fixed order does not increase the smallest number of moves needed. The reason is that the order fixes which operation we use, but not which numbers we use it on. Move 6 has to be a multiplication, and we may multiply any two numbers made so far, which turns out to be freedom enough. The problem was posed by the first author. Our main result is the following.

\begin{theorem}\label{thm:main}
For every $N \ge 9$, starting from 2 with the cycle $+, \times, -, \div$, all of $1, \dots, N$ can be made in exactly $N-1$ moves, and no fewer is possible.
\end{theorem}

The bound $N \ge 9$ cannot be lowered (Section~\ref{sec:small}).

Related problems have been studied as \emph{addition chains}~\cite{DLS}, where numbers are built only by adding, and \emph{$B$-chains}~\cite{BN}, where each step may use any operation from a fixed set $B$. Here the target is a whole interval rather than one number, and the operation at each step is fixed in advance. Finding shortest chains is hard in general. Downey, Leong and Sethi showed that the addition-sequence problem, of reaching several given numbers as cheaply as possible, is NP-complete~\cite{DLS}. That difficulty comes from targets that are few and arbitrary, so that a short chain must be searched for and its optimality argued afterwards. Here the target is a complete interval, and the counting bound of Section~\ref{sec:lower} settles optimality: a sequence that wastes no move meets it, so it is optimal by construction and no search is needed to prove it. Addition chains are treated at length by Knuth~\cite{Knuth}. A related question is integer complexity, the fewest 1s needed to write $n$ with $+$ and $\times$~\cite{OEIS}, which again fixes one target and leaves the operations free.

The lower bound is Lemma~\ref{lem:lower}. Sections~\ref{sec:fails} and~\ref{sec:timing} explain why the interval has to be completed a block at a time, and when each block can start. Section~\ref{sec:tripling} extends $\{1, \dots, P\}$ to $\{1, \dots, 3P\}$, Sections~\ref{sec:everyP} and~\ref{sec:everyN} adjust the last few moves to reach every $N \ge 34$, and Section~\ref{sec:small} finishes the proof with the small cases.

\section{The Rules}\label{sec:rules}

Here are the rules in one place, so we can point back to them later. The game begins with a single number $s$, the \emph{starting number}.
\begin{enumerate}
  \item At each move, choose two numbers already made (possibly the same one twice).
  \item Apply the operation required by the cycle. A division must give a positive integer.
  \item Add the result to our collection of numbers. Numbers are never used up, so we can reuse them as often as we like.
\end{enumerate}
We write $\op{A}, \op{M}, \op{S}, \op{D}$ for addition, multiplication, subtraction and division, and number the moves $1, 2, 3, \dots$. Formally, the operation $\circ_m \in \{+, \times, -, \div\}$ applied at move $m \in \mathbb{Z}^+$ is governed by the 4-periodic schedule
\[
  \circ_m =
  \begin{cases}
    + & \text{if } m \equiv 1 \pmod 4,\\
    \times & \text{if } m \equiv 2 \pmod 4,\\
    - & \text{if } m \equiv 3 \pmod 4,\\
    \div & \text{if } m \equiv 0 \pmod 4.
  \end{cases}
\]
We call this the cycle AMSD. The lower bound in Section~\ref{sec:lower} holds for any starting number $s \in \{1, \dots, N\}$. All later results use $s = 2$.

\section{Lower Bound and a First Example}\label{sec:lower}

Let's start with the easy direction: the fewest moves we could possibly need. We also give a name to the sequences that get away with exactly that many.

\begin{lemma}[Lower bound]\label{lem:lower}
If $s \in \{1, \dots, N\}$, then making all of $1, \dots, N$ starting from $s$ takes at least $N-1$ moves.
\end{lemma}

\begin{proof}
We start with 1 number (the number $s$), so $N-1$ numbers are missing from $\{1, \dots, N\}$. Since each move creates at most 1 new number, reaching $N$ numbers requires at least $N-1$ moves.
\end{proof}

\begin{definition}
A sequence of moves is \emph{perfect} if it makes all of $1, \dots, N$ in exactly $N-1$ moves. In a perfect sequence, every single move must make a previously unconstructed number between 1 and $N$.
\end{definition}

For $N = 9$, starting from 2, a perfect sequence uses $N - 1 = 8$ moves (Table~\ref{tab:n9}).

\begin{table}[ht]
\centering
\begin{tabular}{cccc}
\toprule
Move & Op & Calculation & New number \\
\midrule
1 & A & $2 + 2$ & $\hl{cblue}{4}$ \\
2 & M & $2 \times \hl{cblue}{4}$ & $\hl{cred}{8}$ \\
3 & S & $\hl{cred}{8} - 2$ & $\hl{cgreen}{6}$ \\
4 & D & $\hl{cgreen}{6}/2$ & $\hl{cmag}{3}$ \\
5 & A & $2 + \hl{cmag}{3}$ & $5$ \\
6 & M & $\hl{cmag}{3} \times \hl{cmag}{3}$ & $\hl{corange}{9}$ \\
7 & S & $\hl{corange}{9} - 2$ & $7$ \\
8 & D & $2/2$ & $1$ \\
\bottomrule
\end{tabular}
\caption{A perfect sequence for $N = 9$. A number used again later is printed in bold colour where it is made and where it is reused.}
\label{tab:n9}
\end{table}

Notice that each operation already has its own job here. Multiplication makes the largest numbers (8, 9), subtraction makes the numbers just below them (6, 7), and division halves an even number to get a small one (3, 1). The general construction gives the operations exactly these jobs.

\section{Why the Conventional \texorpdfstring{$n \to n+1$}{n -> n+1} Induction Fails}\label{sec:fails}

The obvious first idea is to grow the interval one number at a time: once we have $\{1, \dots, P\}$, make $P+1$, then $P+2$, and so on. It feels natural, but it breaks the moment a completed interval is followed by a division.

Suppose we have completed $\{1, \dots, P\}$ and the cycle calls for a division $a/b$. Since $a$ and $b$ must come from $\{1, \dots, P\}$:
\begin{itemize}
  \item If $b = 1$, the result is $a/1 = a \le P$, which we already have.
  \item If $b \ge 2$, the result is at most $a/2 \le P/2 \le P$, which we also already have.
\end{itemize}
Either way, division can only give us a number we already have. With $\{1, \dots, 10\}$, for example, the biggest quotient we can make is $10/2 = 5$.

A perfect sequence for $P$ ends on move $P-1$. The next move (move $P$) is a division exactly when $P \equiv 0 \pmod 4$. Hence, in a perfect sequence for any $N > P$ with $P \equiv 0 \pmod 4$, the collection after move $P-1$ can never be exactly $\{1, \dots, P\}$, because move $P$ would then make nothing new. So one-at-a-time growth fails every time $P \equiv 0 \pmod 4$.

So division only helps if we have already built some numbers above $P$ by the time it arrives, because halving those lands back inside the gaps below. From $\{1, \dots, 10\}$, division cannot reach 13, but once 26 has been built, $26/2 = 13$ fills the gap. This is why the interval cannot stay complete after every move. It has to be completed a whole block at a time, with the numbers above $P$ made first, and Section~\ref{sec:factor} works out exactly how big that block has to be.

\section{The Timing Condition}\label{sec:timing}

Before extending $\{1, \dots, P\}$, we need to know which operation the cycle hands us next. Let $P$ be written in the form $P = 4q + 2$, where $q \ge 2$ is an integer that counts how many full 4-move cycles were completed when building $P$.

A perfect sequence for $P$ takes $P - 1$ moves. Substituting $P = 4q + 2$, the final move of the sequence is move number
\[
  P - 1 = (4q + 2) - 1 = 4q + 1 \equiv 1 \pmod 4.
\]
Moves in position $1 \pmod 4$ are additions, so the sequence for $P$ ends on an A, and the very next move, move $P = 4q + 2 \equiv 2 \pmod 4$, is a multiplication (M).

So every extension of $P$ starts on M and runs in 4-move rounds in the repeating order
\[
  \op{M},\ \op{S},\ \op{D},\ \op{A}.
\]
This is just the AMSD cycle read from the multiplication. That is the only job of the condition $P \equiv 2 \pmod 4$: it makes sure every extension opens with a multiplication, which, as we will see, is exactly the move we want first. The base numbers $10, 14, 18, 22$ and $26$ all satisfy $P \equiv 2 \pmod 4$.

\section{Tripling: From \texorpdfstring{$P$}{P} to \texorpdfstring{$3P$}{3P}}\label{sec:tripling}

Fix $P = 4q + 2$ with $P \ge 10$, and assume we have a perfect sequence making $\{1, \dots, P\}$. Two numbers will act as landmarks throughout:
\begin{align}
  \text{midpoint} \quad H &= \tfrac{3}{2}P = \tfrac{3}{2}(4q + 2) = 6q + 3, \label{eq:H}\\
  \text{expanded top} \quad 3P &= 2H = 2(6q + 3) = 12q + 6. \label{eq:top}
\end{align}
Since $P$ is even, $H$ is a whole number. Here is why $H$ matters. The largest number we will ever have is $3P$, and a useful division at best halves it, so no division can ever land above $3P/2 = H$. Think of $H$ as a ceiling for division: everything above it has to be built some other way.

\subsection{Why the expansion factor \texorpdfstring{$k$}{k} is at least three}\label{sec:factor}

Why triple? Why not just double the interval to $2P$, or stretch it by some other factor $k$? The idea is to count how many moves are allowed to build the top part of the new interval, and compare that with how many numbers live there.

The bound below concerns only extensions of this kind, from one completed interval $\{1, \dots, P\}$ to another completed interval $\{1, \dots, kP\}$. It says nothing about perfect sequences that never pass through a completed interval.

Extending $\{1, \dots, P\}$ to $\{1, \dots, kP\}$ requires making $(k-1)P$ new numbers in $m = (k-1)P$ moves. These moves start on M and run in rounds $\op{M}, \op{S}, \op{D}, \op{A}$, so every complete round contains exactly 1 division and 3 non-division moves. Among $m$ consecutive moves starting on M, the divisions occur at positions $3, 7, 11, \dots$, so there are at most $\frac{3m+2}{4}$ non-division moves $(\op{M}, \op{S}, \op{A})$.

Since a division that makes a new number must divide by $b \ge 2$ (dividing by 1 only repeats its numerator) and the maximum value made is $kP$, division can never produce a number in the upper block $(kP/2,\, kP]$. Every number in that upper block must be created by one of the three non-division moves. Comparing the non-division budget for $m = (k-1)P$ moves with the size of the upper block (which contains at least $kP/2$ numbers) gives
\[
  \frac{3(k-1)P + 2}{4} \;\ge\; \frac{kP}{2}.
\]
Multiplying both sides by 4 and expanding,
\[
  3(k-1)P + 2 \ge 2kP \implies 3kP - 3P + 2 \ge 2kP \implies (k-3)P \ge -2.
\]
Dividing by $P$,
\[
  k - 3 \;\ge\; -\frac{2}{P} \implies k \;\ge\; 3 - \frac{2}{P}.
\]
For an integer factor $k$ and $P \ge 10$, this requires $k \ge 3$.

Doubling ($k = 2$) fails concretely: the upper block $(P, 2P]$ requires $P = 4q + 2$ numbers, but the $P$ moves of the extension contain only $3q + 2$ non-division moves, leaving us $q$ moves short. Thus $k = 3$ is the smallest factor for which an extension from $\{1, \dots, P\}$ to $\{1, \dots, kP\}$ is allowed by move arithmetic.

At $k = 3$, the $2P$ new moves form exactly $P/2$ complete rounds, and the budget balances exactly, with no move to spare: $\frac{3}{4}(2P) = \frac{3P}{2}$ non-division moves, matching the $|Y| = \frac{3P}{2}$ elements in the upper block. Furthermore, $kP = 3(4q + 2) = 12q + 6 \equiv 2 \pmod 4$, which preserves the timing condition $P \equiv 2 \pmod 4$ for future extensions.

\subsection{Deriving the tripling construction}\label{sec:derive}

Before writing down any formulas, here is where the construction came from. When we looked at the moves our search program found for extending a completed interval $\{1, \dots, P\}$, the same three habits kept showing up:
\begin{enumerate}
  \item Multiplications consistently multiplied 3 by decreasing integers $P, P-1, P-2, \dots$ to establish high anchor values across $Y$.
  \item Subtractions and additions immediately attached $\pm 1$ around these multiples of 3, building contiguous triples $\{3k-1, 3k, 3k+1\}$.
  \item Division steps always halved a value already made earlier in the extension to fill lower targets in $X$.
\end{enumerate}

\paragraph{The conceptual derivation.}
The upper block $Y = \{H+1, \dots, 3P\}$ contains $\frac{3P}{2}$ numbers unreachable by division. Since each 4-move round supplies 3 non-division moves $(\op{M}, \op{S}, \op{A})$, each round must construct exactly 3 numbers in $Y$. To do this systematically:
\begin{itemize}
  \item For multiplication (M), we select anchor values $C_j = 3(P - j)$ spaced by 3. Since $3 \le P$ and $P - j \le P$, the operation $3 \times (P - j)$ is always valid using available numbers.
  \item For subtraction (S) and addition (A), we use the available number 1 to compute $C_j - 1$ and $C_j + 1$. Together with $C_j$, this forms the triple $\{C_j - 1, C_j, C_j + 1\}$.
  \item For division (D), reaching a target element $x = H - j$ in the lower block $X$ requires halving its double:
  \[
    2(H - j) = 2\Bigl(\tfrac{3P}{2} - j\Bigr) = 3P - 2j.
  \]
  This is always even, since $P$ is even, so the division is exact.
\end{itemize}
Here is the key observation. As $j$ goes up by 1, the number a division needs, $3P - 2j$, drops by 2, while the anchors $C_j = 3P - 3j$ drop by 3. The anchors move down faster than the demand, so the number a division needs always lands on or next to an anchor $C_k$ with $k \le j$, which means it is made no later than round $j$. Section~\ref{sec:lowerX} checks the order of moves inside a round.

\subsection{Partitioning into two blocks}\label{sec:partition}

To turn this balance into an actual construction, we split the $2P$ new numbers at the ceiling $H$:
\[
  X = \{P+1, \dots, H\}, \qquad Y = \{H+1, \dots, 3P\}.
\]
Using $P = 4q + 2$ and $H = 6q + 3$, we check the size of each block:
\begin{align*}
  |X| &= H - P = (6q + 3) - (4q + 2) = 2q + 1 = \tfrac{P}{2},\\
  |Y| &= 3P - H = (12q + 6) - (6q + 3) = 6q + 3 = \tfrac{3P}{2}.
\end{align*}
Here $X$ is the lower block (numbers reachable by division) containing $P/2$ elements, and $Y$ is the upper block (numbers unreachable by division) containing $3P/2$ elements.

In each of the $P/2$ extension rounds, the three non-division moves $(\op{M}, \op{S}, \op{A})$ construct three numbers in $Y$. Halving a number in $Y$ yields a value in $(H/2, H]$, an interval that covers $X$. Thus each round's division takes a number already built in $Y$ and places it into $X$, as illustrated in Figure~\ref{fig:blocks}.

\begin{figure}[ht]
\centering
\small
\begin{tikzpicture}
  \fill[gray!15] (0,0) rectangle (3,0.5);
  \draw (0,0) rectangle (3,0.5);
  \draw (3,0) rectangle (5,0.5);
  \draw (5,0) rectangle (11,0.5);
  \node at (1.5,0.25) {already made};
  \node at (4,0.25) {$X$};
  \node at (8,0.25) {$Y$};
  \node[below] at (0,0) {$1$};
  \node[below] at (3,0) {$P$};
  \node[below] at (5,0) {$H$};
  \node[below] at (11,0) {$3P$};
  \draw[decorate,decoration={brace,amplitude=4pt}] (3.05,0.65) -- (4.95,0.65)
    node[midway,above=4pt,align=center,font=\footnotesize] {$P/2$ numbers\\ one per round};
  \draw[decorate,decoration={brace,amplitude=4pt}] (5.05,0.65) -- (10.95,0.65)
    node[midway,above=4pt,align=center,font=\footnotesize] {$3P/2$ numbers\\ three per round: $\op{M}, \op{S}, \op{A}$};
  \draw[-{Stealth[length=4pt]}] (8,-0.55) to[bend left=25] (4,-0.55);
  \node[font=\footnotesize] at (6,-1.25) {D: halve a number of $Y$ to reach $X$};
\end{tikzpicture}
\caption{The extension from $P$ to $3P$. Three moves of each round work in $Y$ and one in $X$.}
\label{fig:blocks}
\end{figure}
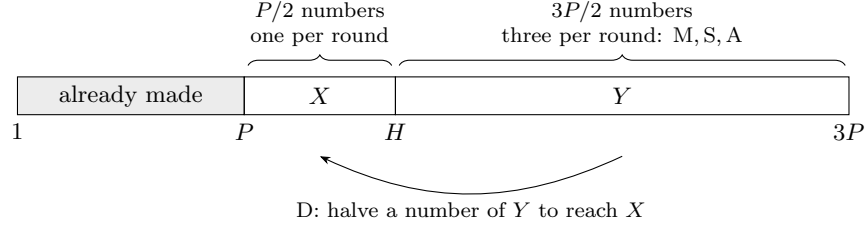

\subsection{Building the upper block \texorpdfstring{$Y$}{Y}}\label{sec:upperY}

To reach the top of $Y$, adding will not do, since two numbers from $\{1, \dots, P\}$ add up to at most $P + P = 2P$. Only a multiplication can jump that high, and multiplying by 3 gives a ladder whose rungs are spaced exactly far enough apart for one subtraction and one addition to fill the gaps.

\paragraph{Definition of \texorpdfstring{$C_j$}{C\_j}.}
For round index $j = 0, 1, \dots, \frac{P}{2} - 1$ (where $\frac{P}{2} - 1 = 2q$), we define
\begin{equation}\label{eq:Cj}
  C_j = 3(P - j) = 3(4q + 2 - j) = 12q + 6 - 3j = 3P - 3j.
\end{equation}
Here $j$ acts as a round counter, so $C_j$ is the $j$-th multiple of 3 counting down from $3P$. Three quick checks show that $C_j$ does what we want:
\begin{enumerate}
  \item Both 3 and $P - j$ lie in $\{1, \dots, P\}$, because $j \le \frac{P}{2} - 1$ implies $P - j \ge \frac{P}{2} + 1 \ge 1$. Thus the multiplication $3 \times (P - j)$ is allowed.
  \item The values are distinct and decrease by 3, since $C_j - C_{j+1} = (3P - 3j) - \bigl(3P - 3(j+1)\bigr) = 3$.
  \item Every $C_j$ lies in $Y$, since the smallest value occurs at $j = \frac{P}{2} - 1 = 2q$, giving $C_{2q} = 3(2q + 2) = 6q + 6 = H + 3 > H$.
\end{enumerate}

Because consecutive $C_j$ differ by 3, the two numbers strictly between $C_{j+1}$ and $C_j$ are $C_{j+1} + 1$ and $C_j - 1$. Subtraction and addition using the available number 1 reach $C_j - 1$ and $C_j + 1$:
\[
  \op{S}:\ C_j - 1, \qquad \op{A}:\ C_j + 1.
\]
So the three non-division moves of round $j \ge 1$ produce the consecutive triple
\[
  T_j = \{C_j - 1,\ C_j,\ C_j + 1\},
\]
and since $C_{j-1} = C_j + 3$, the triple $T_{j-1}$ begins right after $T_j$ ends. The triples therefore tile an interval:
\[
  \bigcup_{j=1}^{P/2-1} T_j = \bigl[C_{P/2-1} - 1,\ C_1 + 1\bigr] = [H+2,\ 3P-2].
\]

Round 0 is the exception. Its multiplication gives $C_0 = 3P$ and its subtraction $3P - 1$. However, $C_0 + 1 = 3P + 1$ exceeds $3P$, so its addition move makes $H + 1$ instead (by adding 1 to $H$, which was made by move D of round 0). Combining these yields
\[
  [H+2,\ 3P-2] \cup \{3P-1,\ 3P\} \cup \{H+1\} = [H+1,\ 3P] = Y,
\]
making each number of $Y$ exactly once. The count agrees, because rounds 1 to $P/2 - 1$ give three numbers each and round 0 gives three, so $3\bigl(\frac{P}{2} - 1\bigr) + 3 = \frac{3P}{2} = |Y|$.

\subsection{Building the lower block \texorpdfstring{$X$}{X}}\label{sec:lowerX}

So where does the number we halve come from? To make $x \in X = \{P+1, \dots, H\}$ by division, we halve its double $2x \in Y$, and we already know from Section~\ref{sec:derive} that this double sits on or next to an anchor. The cases below just pin down which one. We parameterize the target as $x = H - j$ for round index $j \in \{0, 1, \dots, \frac{P}{2} - 1\}$. Substituting $H = 6q + 3$, the required double is
\begin{equation}\label{eq:double}
  2x = 2(H - j) = 2(6q + 3 - j) = 12q + 6 - 2j = 3P - 2j.
\end{equation}
To see exactly which multiple $C_k$ or neighbour $C_k \pm 1$ supplies $3P - 2j$, write $j$ as $3t$, $3t+1$ or $3t+2$ with $t \ge 0$, according to its remainder modulo 3:
\begin{itemize}
  \item If $j = 3t$, then
  \[
    3P - 2j = 3P - 6t = 3(P - 2t) = C_{2t},
  \]
  so the number to halve is exactly the multiplier $C_{2t}$.
  \item If $j = 3t + 1$, then
  \[
    3P - 2j = 3P - 6t - 2 = (3P - 6t - 3) + 1 = 3\bigl(P - (2t+1)\bigr) + 1 = C_{2t+1} + 1,
  \]
  so the number to halve is $C_{2t+1} + 1$, created by the addition move of round $2t + 1$.
  \item If $j = 3t + 2$, then
  \[
    3P - 2j = 3P - 6t - 4 = (3P - 6t - 3) - 1 = 3\bigl(P - (2t+1)\bigr) - 1 = C_{2t+1} - 1,
  \]
  so the number to halve is $C_{2t+1} - 1$, created by the subtraction move of round $2t + 1$.
\end{itemize}
Since $2t \le 3t$ and $2t + 1 \le 3t + 1$, with equality only when $t = 0$, every numerator is made in an earlier round, except in rounds 0 and 1. Round 0 causes no trouble, because its multiplication makes $C_0$ before its division needs it.

There is one snag. Inside a round the order is $\op{M}, \op{S}, \op{D}, \op{A}$, so in round 1 ($j = 1$, so $t = 0$) the division would need $C_1 + 1$ one move before the addition makes it. The fix is to swap what rounds 1 and 2 divide:
\begin{itemize}
  \item In round 1, move D makes $H - 2$ by halving $C_1 - 1 = 3P - 4$ (made by move S of round 1).
  \item In round 2, move D makes $H - 1$ by halving $C_1 + 1 = 3P - 2$ (made by move A of round 1).
\end{itemize}
From round 3 onward, round $j$ makes $H - j$ directly, and the divisions yield every element of $X$:
\[
  H,\ H - 2,\ H - 1,\ H - 3,\ \dots,\ H - \bigl(\tfrac{P}{2} - 1\bigr) = P + 1.
\]

\subsection{Summary of the tripling construction}\label{sec:summary}

Table~\ref{tab:rounds} collects the four moves of each round, and Table~\ref{tab:divisions} lists the divisions in full, since they are the only moves whose form changes from round to round.

\begin{table}[ht]
\centering
\begin{tabular}{ccc}
\toprule
Op & Round 0 & Round $j \ge 1$ \\
\midrule
M & $3 \times P = 3P$ & $3 \times (P - j) = C_j$ \\
S & $3P - 1$ & $C_j - 1$ \\
D & $3P/2 = H$ & half of $C_k$ or $C_k \pm 1$ (Table~\ref{tab:divisions}) \\
A & $H + 1$ & $C_j + 1$ \\
\bottomrule
\end{tabular}
\caption{The four moves of each round of the tripling step.}
\label{tab:rounds}
\end{table}

\begin{table}[ht]
\centering
\begin{tabular}{ccccc}
\toprule
Round & Number halved & Written as & Made in round & Result \\
\midrule
$0$ & $3P$ & $C_0$ & $0$ (M) & $H$ \\
$1$ & $3P - 4$ & $C_1 - 1$ & $1$ (S) & $H - 2$ \\
$2$ & $3P - 2$ & $C_1 + 1$ & $1$ (A) & $H - 1$ \\
\addlinespace
$3t$ & $3P - 6t$ & $C_{2t}$ & $2t$ (M) & $H - 3t$ \\
$3t + 1$ & $3P - 6t - 2$ & $C_{2t+1} + 1$ & $2t + 1$ (A) & $H - 3t - 1$ \\
$3t + 2$ & $3P - 6t - 4$ & $C_{2t+1} - 1$ & $2t + 1$ (S) & $H - 3t - 2$ \\
\bottomrule
\end{tabular}
\caption{The division in each round of the tripling step, for any $P \equiv 2 \pmod 4$ with $P \ge 10$ and $t \ge 1$.}
\label{tab:divisions}
\end{table}

\begin{table}[ht]
\centering
\begin{tabular}{ccccc}
\toprule
Round & M & S & D & A \\
\midrule
0 & $3 \times 10 = \hl{cred}{30}$ & $30 - 1 = 29$ & $\hl{cred}{30}/2 = 15$ & $15 + 1 = 16$ \\
1 & $3 \times 9 = 27$ & $27 - 1 = \hl{cblue}{26}$ & $\hl{cblue}{26}/2 = 13$ & $27 + 1 = \hl{cgreen}{28}$ \\
2 & $3 \times 8 = \hl{cmag}{24}$ & $24 - 1 = 23$ & $\hl{cgreen}{28}/2 = 14$ & $24 + 1 = 25$ \\
3 & $3 \times 7 = 21$ & $21 - 1 = 20$ & $\hl{cmag}{24}/2 = 12$ & $21 + 1 = \hl{corange}{22}$ \\
4 & $3 \times 6 = 18$ & $18 - 1 = 17$ & $\hl{corange}{22}/2 = 11$ & $18 + 1 = 19$ \\
\bottomrule
\end{tabular}
\caption{The tripling step from $P = 10$ to $3P = 30$.}
\label{tab:p10}
\end{table}

\begin{example}
For $P = 10$ ($q = 2$) we have $H = 15$ and $3P = 30$. Starting after the base sequence for 10 (Appendix~\ref{app:base}), the tripling step is shown in Table~\ref{tab:p10}. Each number that a later division halves is printed in bold colour where it is made and where it is halved. Together, the moves make $11, \dots, 30$, each once.
\end{example}

\begin{lemma}[Tripling]\label{lem:tripling}
If $P \equiv 2 \pmod 4$, $P \ge 10$, and $P$ has a perfect sequence from 2, then so does $3P$.
\end{lemma}

\begin{proof}
The $2P$ new moves follow the AMSD cycle. Every operand is either in $\{1, \dots, P\}$ or was made earlier, and every division is exact. Moves $\op{M}, \op{S}, \op{A}$ generate $Y$ and moves D generate $X$, each number once. So after $(P - 1) + 2P = 3P - 1$ moves, we have constructed all of $\{1, \dots, 3P\}$.
\end{proof}

\section{Reaching Every \texorpdfstring{$P \equiv 2 \pmod 4$}{P = 2 mod 4}}\label{sec:everyP}

On its own, tripling only reaches $10 \to 30 \to 90 \to 270 \to \cdots$, which leaves out almost everything. To reach every $P \equiv 2 \pmod 4$, we change the last few moves using small adjustments we call \emph{patches}, and run the induction downwards.

\subsection{Reducing \texorpdfstring{$q$}{q}}\label{sec:reduce}

Let $P = 4q + 2$. Tripling $P$ yields $3P = 3(4q + 2) = 12q + 6 = 4(3q + 1) + 2$. If we add 0, 4 or 8 extra numbers at the end, the target becomes
\begin{alignat*}{2}
  3P &= 12q + 6 = 4(3q + 1) + 2 &\qquad& \implies \quad q' = 3q + 1,\\
  3P + 4 &= 12q + 10 = 4(3q + 2) + 2 && \implies \quad q' = 3q + 2,\\
  3P + 8 &= 12q + 14 = 4(3q + 3) + 2 && \implies \quad q' = 3q + 3.
\end{alignat*}
Reading this backwards: given any target parameter $q \ge 7$, we can write $q = 3r + a$ where $a \in \{1, 2, 3\}$ and $r = \frac{q - a}{3}$. Notice that $2 \le r < q$. The smaller parent interval has parameter $P' = 4r + 2$, and extending $P'$ by tripling (plus 0, 4 or 8) reaches $P = 4q + 2$.

Since $q$ strictly decreases at each step, repeating this reduction eventually lands on some $q \in \{2, 3, 4, 5, 6\}$, corresponding to five base cases
\[
  P \in \{10, 14, 18, 22, 26\}.
\]
Explicit perfect sequences for these five base cases are written out in Appendix~\ref{app:base}.

\subsection{Derivation of a representative patch (\texorpdfstring{$c = 2$}{c = 2})}\label{sec:patch2}

The patches can look like lists of lucky moves, so let's build one from scratch to show that they are not. Take the offset $c = 2$.

\paragraph{Goal.}
Extend a completed interval $\{1, \dots, P\}$ to $3P + 2$, where $P = 4q + 2$, $H = 6q + 3$ and $T = 3P = 12q + 6$.

\paragraph{Step 1: identify missing and target elements.}
We delete the final round of tripling (4 moves). This frees 4 moves and omits the 4 numbers that the deleted round would have made:
\[
  P + 1 = 4q + 3, \quad H + 2 = 6q + 5, \quad H + 3 = 6q + 6, \quad H + 4 = 6q + 7.
\]
We must also construct the 2 new top elements
\[
  3P + 1 = 12q + 7, \qquad 3P + 2 = 12q + 8.
\]
In total, we must generate these 6 distinct elements using $4 + 2 = 6$ moves running in the scheduled order $\op{M}, \op{S}, \op{D}, \op{A}, \op{M}, \op{S}$.

\paragraph{Step 2: derive each move.}
\begin{enumerate}
  \item Move 1 (M) makes $H + 3 = 6q + 6$ by factoring $6q + 6 = 3 \times (2q + 2)$. For $q \ge 2$, both $3 \le P$ and $2q + 2 \le 4q + 2 = P$ are available in $\{1, \dots, P\}$.
  \item Move 2 (S) makes $H + 2 = 6q + 5$ via $(H + 3) - 1$, using $H + 3$ from move 1 and the available number 1.
  \item Move 3 (D) makes $P + 1 = 4q + 3$. Division must halve its double $2(P + 1) = 8q + 6 = 2P + 2$, so we compute $(8q + 6)/2$. The numerator $8q + 6$ was made in an earlier, kept tripling round.
  \item Move 4 (A) makes $H + 4 = 6q + 7$ via $(H + 3) + 1$.
  \item Move 5 (M) makes $3P + 2 = 12q + 8$ by factoring $4 \times (3q + 2)$. Both $4 \le P$ and $3q + 2 \le P$ are available for $q \ge 2$.
  \item Move 6 (S) makes $3P + 1 = 12q + 7$ via $(3P + 2) - 1$, using $3P + 2$ from move 5.
\end{enumerate}
Every line of the $c = 2$ table comes from one of these six steps.

\subsection{Patches and the proof of Lemma~\texorpdfstring{\ref{lem:everyP}}{3}}\label{sec:patches}

Figure~\ref{fig:patch} shows how a patch fits into the sequence. Applying the derivation above across all eleven residue offsets yields the complete patch tables in Appendix~\ref{app:patches}; Tables~\ref{tab:c2} and~\ref{tab:c4} show the two used most often.

\begin{lemma}\label{lem:everyP}
Every $P \ge 10$ with $P \equiv 2 \pmod 4$ has a perfect sequence from 2.
\end{lemma}

\begin{proof}
By induction on $q$, where $P = 4q + 2$. For $q \le 6$, $P$ is one of the base cases in Appendix~\ref{app:base}. For $q \ge 7$, write $q = 3r + a$ with $a \in \{1, 2, 3\}$ and $2 \le r < q$. By induction, $P' = 4r + 2$ has a perfect sequence, and $P = 3P' + 4(a - 1)$. Lemma~\ref{lem:tripling} or the patch for $c = 4$ or $c = 8$ completes the step. Soundness of every patch is proved in Appendix~\ref{app:proof} (Lemma~\ref{lem:patch}).
\end{proof}

\begin{figure}[ht]
\centering
\footnotesize
\begin{tikzpicture}
  \fill[gray!15] (0,0) rectangle (2.4,0.6);
  \draw (0,0) rectangle (2.4,0.6);
  \draw (2.4,0) rectangle (6.4,0.6);
  \draw[dashed] (6.4,0) -- (8.9,0) (6.4,0.6) -- (8.9,0.6);
  \draw (8.9,0) rectangle (12,0.6);
  \node at (1.2,0.3) {sequence for $P$};
  \node at (4.4,0.3) {rounds 0 to $P/2 - 2$};
  \node at (7.65,0.3) {last round};
  \node at (10.45,0.3) {patch};
  \node[below=2pt,align=center,font=\scriptsize] at (1.2,0) {$P - 1$ moves\\ makes $1, \dots, P$};
  \node[below=2pt,align=center,font=\scriptsize] at (4.4,0) {$2P - 4$ moves\\ makes $P + 1, \dots, 3P$ apart\\ from the four at the right};
  \node[below=2pt,align=center,font=\scriptsize] at (7.65,0) {4 moves, deleted;\\ would make $P + 1$,\\ $H + 2$, $H + 3$, $H + 4$};
  \node[below=2pt,align=center,font=\scriptsize] at (10.45,0) {$4 + c$ moves\\ makes those four and\\ $3P + 1, \dots, 3P + c$};
  \draw[-{Stealth[length=4pt]}] (7.65,0.7) to[bend left=45]
    node[midway,above,font=\scriptsize] {replaced by} (10.45,0.7);
\end{tikzpicture}
\caption{Building a sequence for $3P + c$. The patch replaces the deleted round.}
\label{fig:patch}
\end{figure}
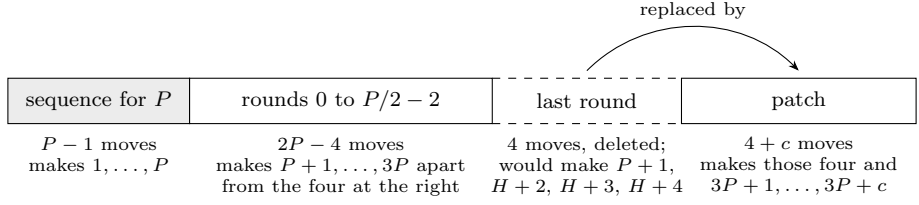

\begin{table}[ht]
\centering
\small
\setlength{\tabcolsep}{3.5pt}
\begin{minipage}[b]{0.48\textwidth}
\centering
\begin{tabular}{llll}
\toprule
Op & Calculation & Result & From \\
\midrule
M & $3 \times (2q + 2)$ & $H + 3$ & \\
S & $(H + 3) - 1$ & $H + 2$ & line 1 \\
D & $(8q + 6)/2$ & $P + 1$ & kept \\
A & $(H + 3) + 1$ & $H + 4$ & line 1 \\
M & $4 \times (3q + 2)$ & $T + 2$ & \\
S & $(T + 2) - 1$ & $T + 1$ & line 5 \\
\bottomrule
\end{tabular}
\end{minipage}\hfill
\begin{minipage}[b]{0.48\textwidth}
\centering
\begin{tabular}{llll}
\toprule
Op & Calculation & Result & From \\
\midrule
M & $3 \times (2q + 2)$ & $H + 3$ & \\
S & $(H + 5) - 1$ & $H + 4$ & kept \\
D & $(8q + 6)/2$ & $P + 1$ & kept \\
A & $T + 4$ & $T + 4$ & kept \\
M & $4 \times (3q + 2)$ & $T + 2$ & \\
S & $(T + 2) - 1$ & $T + 1$ & line 5 \\
D & $(T + 4)/2$ & $H + 2$ & line 4 \\
A & $(T + 2) + 1$ & $T + 3$ & line 5 \\
\bottomrule
\end{tabular}
\end{minipage}

\medskip
\begin{minipage}[t]{0.48\textwidth}
\centering
\captionof{table}{Patch for $c = 2$ (6 moves).}
\label{tab:c2}
\end{minipage}\hfill
\begin{minipage}[t]{0.48\textwidth}
\centering
\captionof{table}{Patch for $c = 4$ (8 moves).}
\label{tab:c4}
\end{minipage}
\end{table}

\section{Reaching Every \texorpdfstring{$N$}{N}}\label{sec:everyN}

Lemma~\ref{lem:everyP} covers $N \equiv 2 \pmod 4$. For all other values of $N$, we write $N = 3P + c$ for a suitable offset $c$ and base parameter $P \equiv 2 \pmod 4$.

Since $P = 4q + 2$, we have $3P = 12q + 6 \equiv 6 \pmod{12}$. Therefore $N = 3P + c \equiv 6 + c \pmod{12}$. For each remainder of $N$ modulo 12, we choose $c$ so that $N - c \equiv 6 \pmod{12}$, as in Table~\ref{tab:offset}.

\begin{table}[ht]
\centering
\setlength{\tabcolsep}{4.5pt}
\begin{tabular}{l*{12}{c}}
\toprule
$N \bmod 12$ & 0 & 1 & 2 & 3 & 4 & 5 & 6 & 7 & 8 & 9 & 10 & 11 \\
\midrule
$c$ & 6 & 7 & 8 & 9 & 10 & 11 & 0 & 1 & 2 & 15 & 4 & 5 \\
\bottomrule
\end{tabular}
\caption{The offset $c$ for each residue of $N$ modulo 12.}
\label{tab:offset}
\end{table}

Then $N - c = 12q + 6$ for an integer $q$, so
\[
  P = \frac{N - c}{3} = \frac{12q + 6}{3} = 4q + 2,
\]
which is a whole number satisfying $P \equiv 2 \pmod 4$.

No patch exists for $c = 3$ when the last one or two rounds are deleted, at least for $2 \le q \le 40$ (Section~\ref{sec:patchsearch}), so we replace it with $c = 15$, which belongs to the same residue class since $15 \equiv 3 \pmod{12}$. The patches for $c = 1$ and $c = 15$ delete the last two rounds (releasing 8 numbers to be re-made) rather than just the final round (Table~\ref{tab:lengths}). All patches are listed in Appendix~\ref{app:patches}.

\begin{table}[ht]
\centering
\setlength{\tabcolsep}{4.5pt}
\begin{tabular}{l*{11}{c}}
\toprule
$c$ & 1 & 2 & 4 & 5 & 6 & 7 & 8 & 9 & 10 & 11 & 15 \\
\midrule
rounds deleted & 2 & 1 & 1 & 1 & 1 & 1 & 1 & 1 & 1 & 1 & 2 \\
patch length & 9 & 6 & 8 & 9 & 10 & 11 & 12 & 13 & 14 & 15 & 23 \\
\bottomrule
\end{tabular}
\caption{Number of deleted rounds and patch length for each offset.}
\label{tab:lengths}
\end{table}

\subsection{Size of \texorpdfstring{$P$}{P}}\label{sec:sizeP}

The patches need $P \ge 10$, and the $c = 1$ patch needs $P \ge 14$. For the smallest $N \ge 34$ in each residue class, Table~\ref{tab:smallestN} gives the resulting $P$. Larger $N$ in the same class give larger $P$.

\begin{table}[ht]
\centering
\setlength{\tabcolsep}{4.5pt}
\begin{tabular}{l*{12}{c}}
\toprule
$N \bmod 12$ & 0 & 1 & 2 & 3 & 4 & 5 & 6 & 7 & 8 & 9 & 10 & 11 \\
\midrule
smallest $N \ge 34$ & 36 & 37 & 38 & 39 & 40 & 41 & 42 & 43 & 44 & 45 & 34 & 35 \\
$P = (N - c)/3$ & 10 & 10 & 10 & 10 & 10 & 10 & 14 & 14 & 14 & 10 & 10 & 10 \\
\bottomrule
\end{tabular}
\caption{The value of $P$ for the smallest $N \ge 34$ in each residue class.}
\label{tab:smallestN}
\end{table}

\begin{theorem}\label{thm:large}
Every $N \ge 34$ has a perfect sequence from 2 with the cycle \textup{AMSD}.
\end{theorem}

\begin{proof}
Let $c$ be the offset for $N \bmod 12$ from Table~\ref{tab:offset}, and let $P = (N - c)/3$. Then $P \equiv 2 \pmod 4$ and $P \ge 10$ (with $P \ge 14$ when $c = 1$). Lemma~\ref{lem:everyP} gives a perfect sequence for $P$. If $c = 0$, Lemma~\ref{lem:tripling} extends it to $3P = N$. Otherwise, Lemma~\ref{lem:patch} in Appendix~\ref{app:proof} shows that applying patch $c$ yields a perfect sequence for $3P + c = N$.
\end{proof}

\subsection{An example: the reduction chain}\label{sec:chain}

The whole proof is a chain of reductions. One patch takes $N$ to $P$, and repeated steps of Lemma~\ref{lem:everyP} take $P$ down to a base case. For $N = 100000$,
\begin{multline*}
  14 \xrightarrow{\times 3} 42 \xrightarrow{+8} 134 \xrightarrow{+8} 410 \xrightarrow{+4} 1234 \\
  \xrightarrow{\times 3} 3702 \xrightarrow{+4} 11110 \xrightarrow{\times 3} 33330 \xrightarrow{+10} 100000.
\end{multline*}
Here $\xrightarrow{\times 3}$ means $P \to 3P$ and $\xrightarrow{+c}$ means $P \to 3P + c$. Reading from right to left, $100000 \equiv 4 \pmod{12}$ gives $c = 10$ and $P = 33330$, after which each $q$ is reduced as in Section~\ref{sec:reduce}. The resulting sequence has 99999 moves.

\section{Small Cases and the Proof of Theorem~\texorpdfstring{\ref{thm:main}}{1}}\label{sec:small}

The general construction requires $P \ge 10$, so it covers all $N \ge 34$ and also $N = 30, 32$ (from $P = 10$ with $c = 0, 2$). The base cases $10, 14, 18, 22, 26$ are listed in Appendix~\ref{app:base}, and the case $N = 9$ is given in Section~\ref{sec:lower}. The remaining cases $9 \le N \le 33$ were solved by computer search and are listed in Appendix~\ref{app:small}.

\begin{proof}[Proof of Theorem~\ref{thm:main}]
Lemma~\ref{lem:lower} shows that at least $N - 1$ moves are needed. For $N \ge 34$, Theorem~\ref{thm:large} gives a perfect sequence. For $9 \le N \le 33$, a perfect sequence is given in Section~\ref{sec:lower} ($N = 9$), Appendix~\ref{app:base} ($N = 10, 14, 18, 22, 26$), Lemma~\ref{lem:tripling} and Lemma~\ref{lem:patch} with $P = 10$ ($N = 30, 32$), or Appendix~\ref{app:small} (all other $N$).
\end{proof}

The bound $N \ge 9$ cannot be lowered. Computer search confirmed that no legal sequence of $N - 1$ moves can construct all of $\{1, \dots, N\}$ starting from 2 for any $2 \le N \le 8$. For each $2 \le N \le 8$ we enumerated every legal sequence of $N - 1$ moves from 2, branching over all ordered pairs of numbers made so far. Discarding a move whose result repeats or leaves $\{1, \dots, N\}$ loses nothing, since a perfect sequence has to make a new number of $\{1, \dots, N\}$ at every move. None of these searches reached the full interval. Perfect sequences for these $N$ do exist if the cycle may begin on $\div$ instead of $+$. Starting from 2, this is enough for $N \le 4$. For $5 \le N \le 8$ the starting number must change as well: to 5 for $N = 5$, to 6 for $N = 6$, and to 6 or 7 for $N = 7, 8$. We do not treat these variants here.

\section{Search, Discovery, and Proof Verification}\label{sec:search}

The proof came from going back and forth between a computer search and working things out by hand:
\begin{enumerate}
  \item We first wrote a search program to enumerate legal move sequences for small $N$. It found the base sequences for $P \in \{10, 14, 18, 22, 26\}$ and the small cases in Appendix~\ref{app:small}, and its output led to the patterns described in Section~\ref{sec:derive}.
  \item When constructing the residue offsets $c$, we could not find a patch for $c = 3$. Noticing that residue $3 \pmod{12}$ can be satisfied by $c = 15$, we derived a two-round replacement patch (23 moves) that bridged the gap. Section~\ref{sec:patchsearch} shows that this was not just bad luck in the search.
  \item Guided by these empirical patterns, we formulated the general tripling construction and patch equations algebraically.
  \item To check every case, we wrote verification scripts in Python that build each sequence from the construction and check it move by move (Section~\ref{sec:verify}).
\end{enumerate}
All of the code needs only Python 3.8 or later and is available at
\[
  \text{\url{https://github.com/poscle/AMSD-Number-Generator}}.
\]

\subsection{Checking the construction}\label{sec:verify}

The program \texttt{amsd.py} builds the sequence for a given $N$ straight from the construction, in the way Section~\ref{sec:computing} describes, and then hands it to a separate checker that knows nothing about how the sequence was made. The checker asks only whether each move uses the operation the cycle calls for, whether both operands were already made, whether a division is exact, and whether the result is new and lies in $\{1, \dots, N\}$, and at the end whether all of $\{1, \dots, N\}$ was made in $N - 1$ moves.

Running \texttt{python amsd.py --check 9 100000} builds and checks every $N$ with $9 \le N \le 100000$, and all 99992 sequences pass. This covers the table cases up to $N = 33$ as well as every $N$ reached by the construction. Since a sequence for $N$ has $N - 1$ moves, the total work grows like $N^2$, and on our test machine (a 2.1 GHz Intel Xeon, Python 3.13) the full range takes about 80 minutes of computing time, which we split into the two ranges $9 \le N \le 70711$ and $70712 \le N \le 100000$ run side by side, finishing in about 45 minutes. A single large value is quick because the construction never searches, so building and checking the 99999-move sequence for $N = 100000$ alone takes about 0.2 seconds. For readers who do not want to wait for the full range, the script \texttt{targeted\_large\_checks.py} checks $N = 5000$, 9999, 10000, 12345, 20000, 35462, 50000, 99991 and 100000 in under a second.

\subsection{Searching for patches}\label{sec:patchsearch}

The patches in Appendix~\ref{app:patches} were derived by hand from search output, so it is fair to ask whether a patch for $c = 3$ exists and we simply missed it. A patch is small enough that we can search for all of them. When a patch deletes the last $d$ rounds, it has to make the $4d$ deleted numbers together with $T + 1, \dots, T + c$ in exactly $4d + c$ moves, and since the sequence is perfect, every move has to make one of these numbers. So at each move the search only chooses which missing number to make next, and checks that two available numbers give it under the operation the cycle requires. Which pair of operands is used does not matter for later moves, because it does not change which numbers have been made, so the search covers every possible patch and not only ones shaped like those in Appendix~\ref{app:patches}. The program \texttt{patch\_search.py} runs it for every $c \in \{1, \dots, 11, 15\}$, $d \in \{1, 2\}$ and $2 \le q \le 40$ in about 80 seconds, with the results in Table~\ref{tab:patchsearch}.

\begin{table}[ht]
\centering
\begin{tabular}{lcc}
\toprule
Offset & $d = 1$ & $d = 2$ \\
\midrule
$c = 1$ & 32 of 39 & 39 of 39 \\
$c = 3$ & 0 of 39 & 0 of 39 \\
$c = 15$ & 0 of 39 & 39 of 39 \\
every other $c \le 11$ & 39 of 39 & 39 of 39 \\
\bottomrule
\end{tabular}
\caption{Number of values $2 \le q \le 40$ for which a patch deleting $d$ rounds exists.}
\label{tab:patchsearch}
\end{table}

The search finds a patch wherever Appendix~\ref{app:patches} has one, which is a check that it works. For $c = 3$ there is no patch for any $q$ in this range, whether one round or two are deleted, and this is why we needed $c = 15$. For $c = 15$ there is no patch that deletes only one round, which is why ours deletes two. The case $c = 1$ is less obvious: one-round patches exist for most $q$ but not for $q = 2, 6, 11, 16, 17, 22$ or $31$, so no single one-round formula can work for every $q$, and the two-round patch is needed there too.

\section{Computing the Sequence}\label{sec:computing}

Since the proof is constructive, it also gives a way to write down a perfect sequence for a given $N \ge 34$:
\begin{enumerate}
  \item Look up $c$ in Table~\ref{tab:offset} and set $P = (N - c)/3$.
  \item Reduce $P$ as in Section~\ref{sec:reduce}, recording at each step whether it was $\times 3$, $+4$ or $+8$, until a base case is reached.
  \item Start from the base sequence in Appendix~\ref{app:base}, apply the recorded steps in reverse order, and finish with the patch for $c$.
\end{enumerate}
Every move in step 3 is given by a formula. In the tripling step, round $j$ is
\begin{center}
\begin{tabular}{ll}
\toprule
Op & Calculation \\
\midrule
M & $3 \times (P - j) = C_j$ \\
S & $C_j - 1$ \\
D & $(3P - 2j)/2 = H - j$ \quad (with rounds 1 and 2 swapped) \\
A & $C_j + 1$ \quad (or $H + 1$ when $j = 0$) \\
\bottomrule
\end{tabular}
\end{center}
and the patches are fixed lists. Because every move comes from a formula, no step has to check which numbers exist, so each move costs a fixed amount of work and the whole sequence takes work proportional to $N$. The tripling step is then a loop of four instructions repeated $P/2$ times, with fixed exceptions only in rounds 0, 1 and 2.

This matters because the sequences are hard to find any other way. A search that tries moves one at a time has to consider every pair of numbers made so far. Before move $m$ there are $m$ numbers, so $m^2$ ordered pairs, and a perfect sequence for $N$ is one of up to $\prod_{m=1}^{N-1} m^2 = \bigl((N-1)!\bigr)^2$ candidates, while the formulas give $N$ in the thousands at once. Knowing the operation in advance is what lets us fix the operands in advance too.

\section{Conclusion}\label{sec:conclusion}

We started with one number, four operations, and a cycle that never lets us choose which operation comes next, and still, for every $N \ge 9$, we reach the lower bound of $N - 1$ moves exactly.

This is possible because each operation fulfils a natural role once $\{1, \dots, P\}$ is available. Multiplication by 3 climbs to the top of the expanded interval, subtraction and addition fill the gaps around it, and division by 2 pulls large values back down to fill the lower block $X$. With $P = 10$ the first round shows all three at work, since $3 \times 10 = 30$ sits at the top, $30 - 1 = 29$ beside it, and $30/2 = 15$ in the middle. The counting in Section~\ref{sec:factor} shows that these jobs fit exactly, three numbers in the upper block and one in the lower block per round, so the construction follows the fixed order directly.

Turning that plan into a proof is what the numbers $C_j$ and the conditions modulo 3, 4 and 12 are for. The argument modulo 3 matters most, since it puts every number a division needs on or next to an anchor $C_k$ made in time (Section~\ref{sec:lowerX}). The other two are simpler, with modulo 4 keeping each extension on a multiplication and modulo 12 telling us which patch to use.

A computer can confirm perfect sequences for thousands of values, but only the construction tells us what happens beyond them, since every $N$ reduces to a base case in finitely many steps. The fixed cycle removes enough freedom that the remaining choices can be written in closed form.

\section{Open Questions}\label{sec:open}

The proofs of Lemmas~\ref{lem:tripling} and~\ref{lem:patch} never use the starting number. They need only a completed interval $\{1, \dots, P\}$, which contains 2 whatever the start, so dividing by 2 is always possible. The starting number enters only through the base cases, and there it does matter: an exhaustive search shows that 2 is the only starting number from which $P = 10$ has a perfect sequence, and the same holds for $P = 14$. Another start would therefore need a different set of base cases. Natural questions are:
\begin{itemize}
  \item Which other starting numbers give perfect sequences for all large $N$?
  \item Which other orders of the four operations work? Showing that some order fails for infinitely many $N$, or classifying the orders that succeed, would turn this from a result about one rule into a statement about which rules can work at all.
  \item What happens if an operation, such as division, is removed?
  \item How many perfect sequences are there for a given $N$?
  \item The patches were found case by case, and $c = 3$ had to be replaced by $c = 15$. Is there a single rule that extends $P$ to $3P + c$ for every $c$?
\end{itemize}

\section*{Use of AI}

The construction and its proofs were developed by the authors. The following tools were used, and the authors reviewed and take responsibility for all of their output.
\begin{itemize}
  \item Claude Opus 5.5 (Anthropic) was used to debug the authors' original Python verification code, which contained errors, into the robust version now in the repository. It was also used to debug \LaTeX\ typesetting, and to suggest revisions to the framing and wording of the expository passages in Sections 2--7, 10 and 11, to make the writing clearer. The mathematics, proofs and tables were not changed by these revisions.
  \item Claude Opus 5.5 (Anthropic) was also used to reconstruct the \LaTeX\ source of this manuscript from its compiled PDF after the original source file was lost, and to write the exhaustive patch search \texttt{patch\_search.py} described in Section~\ref{sec:patchsearch}. The authors ran the search and checked its results.
  \item Overleaf's integrated AI editing tools and Grammarly Pro were used for grammar refinement and editorial polishing.
\end{itemize}
The authors checked the terms of use of each tool and confirm that this use is suitable for publication.

\appendix

\section{Base Sequences}\label{app:base}

All five start from 2 and follow AMSD, reading down each column.

\par\medskip
{\small\noindent
\makebox[\textwidth][c]{%
\makebox[0.2\textwidth][c]{%
\begin{tabular}[t]{@{}r@{\,}c@{\,}l@{\,}c@{\,}l@{}}
\multicolumn{5}{@{}c@{}}{$P = 10$} \\
$2$ & $+$ & $2$ & $=$ & $4$ \\
$2$ & $\times$ & $4$ & $=$ & $8$ \\
$8$ & $-$ & $2$ & $=$ & $6$ \\
$8$ & $/$ & $8$ & $=$ & $1$ \\
$1$ & $+$ & $4$ & $=$ & $5$ \\
$2$ & $\times$ & $5$ & $=$ & $10$ \\
$8$ & $-$ & $1$ & $=$ & $7$ \\
$6$ & $/$ & $2$ & $=$ & $3$ \\
$1$ & $+$ & $8$ & $=$ & $9$ \\
\end{tabular}}%
\makebox[0.2\textwidth][c]{%
\begin{tabular}[t]{@{}r@{\,}c@{\,}l@{\,}c@{\,}l@{}}
\multicolumn{5}{@{}c@{}}{$P = 14$} \\
$2$ & $+$ & $2$ & $=$ & $4$ \\
$2$ & $\times$ & $4$ & $=$ & $8$ \\
$8$ & $-$ & $2$ & $=$ & $6$ \\
$8$ & $/$ & $8$ & $=$ & $1$ \\
$1$ & $+$ & $2$ & $=$ & $3$ \\
$2$ & $\times$ & $6$ & $=$ & $12$ \\
$12$ & $-$ & $2$ & $=$ & $10$ \\
$10$ & $/$ & $2$ & $=$ & $5$ \\
$2$ & $+$ & $12$ & $=$ & $14$ \\
$3$ & $\times$ & $3$ & $=$ & $9$ \\
$12$ & $-$ & $1$ & $=$ & $11$ \\
$14$ & $/$ & $2$ & $=$ & $7$ \\
$1$ & $+$ & $12$ & $=$ & $13$ \\
\end{tabular}}%
\makebox[0.2\textwidth][c]{%
\begin{tabular}[t]{@{}r@{\,}c@{\,}l@{\,}c@{\,}l@{}}
\multicolumn{5}{@{}c@{}}{$P = 18$} \\
$2$ & $+$ & $2$ & $=$ & $4$ \\
$2$ & $\times$ & $4$ & $=$ & $8$ \\
$8$ & $-$ & $2$ & $=$ & $6$ \\
$8$ & $/$ & $8$ & $=$ & $1$ \\
$1$ & $+$ & $2$ & $=$ & $3$ \\
$2$ & $\times$ & $6$ & $=$ & $12$ \\
$12$ & $-$ & $2$ & $=$ & $10$ \\
$10$ & $/$ & $2$ & $=$ & $5$ \\
$1$ & $+$ & $10$ & $=$ & $11$ \\
$2$ & $\times$ & $8$ & $=$ & $16$ \\
$16$ & $-$ & $2$ & $=$ & $14$ \\
$14$ & $/$ & $2$ & $=$ & $7$ \\
$1$ & $+$ & $12$ & $=$ & $13$ \\
$3$ & $\times$ & $6$ & $=$ & $18$ \\
$16$ & $-$ & $1$ & $=$ & $15$ \\
$18$ & $/$ & $2$ & $=$ & $9$ \\
$1$ & $+$ & $16$ & $=$ & $17$ \\
\end{tabular}}%
\makebox[0.2\textwidth][c]{%
\begin{tabular}[t]{@{}r@{\,}c@{\,}l@{\,}c@{\,}l@{}}
\multicolumn{5}{@{}c@{}}{$P = 22$} \\
$2$ & $+$ & $2$ & $=$ & $4$ \\
$2$ & $\times$ & $4$ & $=$ & $8$ \\
$8$ & $-$ & $2$ & $=$ & $6$ \\
$8$ & $/$ & $8$ & $=$ & $1$ \\
$1$ & $+$ & $2$ & $=$ & $3$ \\
$2$ & $\times$ & $6$ & $=$ & $12$ \\
$12$ & $-$ & $2$ & $=$ & $10$ \\
$10$ & $/$ & $2$ & $=$ & $5$ \\
$1$ & $+$ & $12$ & $=$ & $13$ \\
$2$ & $\times$ & $8$ & $=$ & $16$ \\
$16$ & $-$ & $2$ & $=$ & $14$ \\
$14$ & $/$ & $2$ & $=$ & $7$ \\
$1$ & $+$ & $14$ & $=$ & $15$ \\
$2$ & $\times$ & $10$ & $=$ & $20$ \\
$20$ & $-$ & $2$ & $=$ & $18$ \\
$18$ & $/$ & $2$ & $=$ & $9$ \\
$2$ & $+$ & $20$ & $=$ & $22$ \\
$3$ & $\times$ & $7$ & $=$ & $21$ \\
$18$ & $-$ & $1$ & $=$ & $17$ \\
$22$ & $/$ & $2$ & $=$ & $11$ \\
$1$ & $+$ & $18$ & $=$ & $19$ \\
\end{tabular}}%
\makebox[0.2\textwidth][c]{%
\begin{tabular}[t]{@{}r@{\,}c@{\,}l@{\,}c@{\,}l@{}}
\multicolumn{5}{@{}c@{}}{$P = 26$} \\
$2$ & $+$ & $2$ & $=$ & $4$ \\
$2$ & $\times$ & $4$ & $=$ & $8$ \\
$8$ & $-$ & $2$ & $=$ & $6$ \\
$8$ & $/$ & $8$ & $=$ & $1$ \\
$1$ & $+$ & $2$ & $=$ & $3$ \\
$2$ & $\times$ & $6$ & $=$ & $12$ \\
$12$ & $-$ & $2$ & $=$ & $10$ \\
$10$ & $/$ & $2$ & $=$ & $5$ \\
$2$ & $+$ & $12$ & $=$ & $14$ \\
$2$ & $\times$ & $8$ & $=$ & $16$ \\
$16$ & $-$ & $1$ & $=$ & $15$ \\
$14$ & $/$ & $2$ & $=$ & $7$ \\
$1$ & $+$ & $16$ & $=$ & $17$ \\
$2$ & $\times$ & $10$ & $=$ & $20$ \\
$20$ & $-$ & $2$ & $=$ & $18$ \\
$18$ & $/$ & $2$ & $=$ & $9$ \\
$1$ & $+$ & $18$ & $=$ & $19$ \\
$2$ & $\times$ & $12$ & $=$ & $24$ \\
$24$ & $-$ & $2$ & $=$ & $22$ \\
$22$ & $/$ & $2$ & $=$ & $11$ \\
$2$ & $+$ & $24$ & $=$ & $26$ \\
$3$ & $\times$ & $7$ & $=$ & $21$ \\
$24$ & $-$ & $1$ & $=$ & $23$ \\
$26$ & $/$ & $2$ & $=$ & $13$ \\
$1$ & $+$ & $24$ & $=$ & $25$ \\
\end{tabular}}%
}
\par}\smallskip

\section{Patches}\label{app:patches}

Throughout, $P = 4q + 2$, $H = 6q + 3$ and $T = 12q + 6 = 3P$. Each patch replaces the final round (or final two rounds) of the tripling construction and continues the cycle from M.

When a patch starts, every number from 1 to $3P$ is available except the ones made by the deleted rounds. Deleting the final round removes $P + 1$, $H + 2$, $H + 3$, $H + 4$, and deleting the final two rounds also removes $P + 2$, $H + 5$, $H + 6$, $H + 7$. So a patch may use, for example, $H + 5$ or $H + 8$ when these were made by a round that is kept.

The last column says where the operands come from. \emph{Kept} means a number made by the part of the tripling construction that is kept, and \emph{line $i$} means the result of line $i$ of the same patch. A blank entry means both operands are in $\{1, \dots, P\}$. Each result is new and lies in $\{1, \dots, 3P + c\}$.

\noindent
\begin{minipage}[t]{0.48\textwidth}
\centering
\parbox{\linewidth}{\centering \textbf{Offset} $c = 1$ (deletes two rounds; 9 moves)}\par\smallskip
\small
\setlength{\tabcolsep}{3.5pt}
\begin{tabular}{rclll}
\toprule
\# & Op & Calculation & Result & From \\
\midrule
1 & M & $3\times(2q+2)$ & $H+3$ &  \\
2 & S & $(H+3)-1$ & $H+2$ & line 1 \\
3 & D & $(8q+6)/2$ & $P+1$ & kept \\
4 & A & $(H+3)+1$ & $H+4$ & line 1 \\
5 & M & $2\times(3q+4)$ & $H+5$ &  \\
6 & S & $(H+8)-2$ & $H+6$ & kept \\
7 & D & $(8q+8)/2$ & $P+2$ & kept \\
8 & A & $T+1$ & $T+1$ & kept \\
9 & M & $2\times(3q+5)$ & $H+7$ &  \\
\bottomrule
\end{tabular}
\end{minipage}\hfill
\begin{minipage}[t]{0.48\textwidth}
\centering
\parbox{\linewidth}{\centering \textbf{Offset} $c = 2$ (deletes one round; 6 moves)}\par\smallskip
\small
\setlength{\tabcolsep}{3.5pt}
\begin{tabular}{rclll}
\toprule
\# & Op & Calculation & Result & From \\
\midrule
1 & M & $3\times(2q+2)$ & $H+3$ &  \\
2 & S & $(H+3)-1$ & $H+2$ & line 1 \\
3 & D & $(8q+6)/2$ & $P+1$ & kept \\
4 & A & $(H+3)+1$ & $H+4$ & line 1 \\
5 & M & $4\times(3q+2)$ & $T+2$ &  \\
6 & S & $(T+2)-1$ & $T+1$ & line 5 \\
\bottomrule
\end{tabular}
\end{minipage}

\bigskip

\noindent
\begin{minipage}[t]{0.48\textwidth}
\centering
\parbox{\linewidth}{\centering \textbf{Offset} $c = 4$ (deletes one round; 8 moves)}\par\smallskip
\small
\setlength{\tabcolsep}{3.5pt}
\begin{tabular}{rclll}
\toprule
\# & Op & Calculation & Result & From \\
\midrule
1 & M & $3\times(2q+2)$ & $H+3$ &  \\
2 & S & $(H+5)-1$ & $H+4$ & kept \\
3 & D & $(8q+6)/2$ & $P+1$ & kept \\
4 & A & $T+4$ & $T+4$ & kept \\
5 & M & $4\times(3q+2)$ & $T+2$ &  \\
6 & S & $(T+2)-1$ & $T+1$ & line 5 \\
7 & D & $(T+4)/2$ & $H+2$ & line 4 \\
8 & A & $(T+2)+1$ & $T+3$ & line 5 \\
\bottomrule
\end{tabular}
\end{minipage}\hfill
\begin{minipage}[t]{0.48\textwidth}
\centering
\parbox{\linewidth}{\centering \textbf{Offset} $c = 5$ (deletes one round; 9 moves)}\par\smallskip
\small
\setlength{\tabcolsep}{3.5pt}
\begin{tabular}{rclll}
\toprule
\# & Op & Calculation & Result & From \\
\midrule
1 & M & $3\times(2q+2)$ & $H+3$ &  \\
2 & S & $(H+5)-1$ & $H+4$ & kept \\
3 & D & $(8q+6)/2$ & $P+1$ & kept \\
4 & A & $T+4$ & $T+4$ & kept \\
5 & M & $4\times(3q+2)$ & $T+2$ &  \\
6 & S & $(T+2)-1$ & $T+1$ & line 5 \\
7 & D & $(T+4)/2$ & $H+2$ & line 4 \\
8 & A & $(T+4)+1$ & $T+5$ & line 4 \\
9 & M & $3\times(P+1)$ & $T+3$ & line 3 \\
\bottomrule
\end{tabular}
\end{minipage}

\bigskip

\noindent
\begin{minipage}[t]{0.48\textwidth}
\centering
\parbox{\linewidth}{\centering \textbf{Offset} $c = 6$ (deletes one round; 10 moves)}\par\smallskip
\small
\setlength{\tabcolsep}{3.5pt}
\begin{tabular}{rclll}
\toprule
\# & Op & Calculation & Result & From \\
\midrule
1 & M & $3\times(2q+2)$ & $H+3$ &  \\
2 & S & $(H+5)-1$ & $H+4$ & kept \\
3 & D & $(8q+6)/2$ & $P+1$ & kept \\
4 & A & $T+1$ & $T+1$ & kept \\
5 & M & $6\times(2q+2)$ & $T+6$ &  \\
6 & S & $(T+6)-2$ & $T+4$ & line 5 \\
7 & D & $(T+4)/2$ & $H+2$ & line 6 \\
8 & A & $(T+1)+1$ & $T+2$ & line 4 \\
9 & M & $3\times(P+1)$ & $T+3$ & line 3 \\
10 & S & $(T+6)-1$ & $T+5$ & line 5 \\
\bottomrule
\end{tabular}
\end{minipage}\hfill
\begin{minipage}[t]{0.48\textwidth}
\centering
\parbox{\linewidth}{\centering \textbf{Offset} $c = 7$ (deletes one round; 11 moves)}\par\smallskip
\small
\setlength{\tabcolsep}{3.5pt}
\begin{tabular}{rclll}
\toprule
\# & Op & Calculation & Result & From \\
\midrule
1 & M & $4\times(3q+2)$ & $T+2$ &  \\
2 & S & $(H+5)-1$ & $H+4$ & kept \\
3 & D & $(8q+6)/2$ & $P+1$ & kept \\
4 & A & $T+1$ & $T+1$ & kept \\
5 & M & $6\times(2q+2)$ & $T+6$ &  \\
6 & S & $(T+6)-2$ & $T+4$ & line 5 \\
7 & D & $(T+4)/2$ & $H+2$ & line 6 \\
8 & A & $(T+6)+1$ & $T+7$ & line 5 \\
9 & M & $3\times(P+1)$ & $T+3$ & line 3 \\
10 & S & $(T+6)-1$ & $T+5$ & line 5 \\
11 & D & $(T+6)/2$ & $H+3$ & line 5 \\
\bottomrule
\end{tabular}
\end{minipage}

\bigskip

\noindent
\begin{minipage}[t]{0.48\textwidth}
\centering
\parbox{\linewidth}{\centering \textbf{Offset} $c = 8$ (deletes one round; 12 moves)}\par\smallskip
\small
\setlength{\tabcolsep}{3.5pt}
\begin{tabular}{rclll}
\toprule
\# & Op & Calculation & Result & From \\
\midrule
1 & M & $4\times(3q+2)$ & $T+2$ &  \\
2 & S & $(H+5)-3$ & $H+2$ & kept \\
3 & D & $(8q+6)/2$ & $P+1$ & kept \\
4 & A & $T+1$ & $T+1$ & kept \\
5 & M & $6\times(2q+2)$ & $T+6$ &  \\
6 & S & $(T+6)-3$ & $T+3$ & line 5 \\
7 & D & $(T+6)/2$ & $H+3$ & line 5 \\
8 & A & $(T+6)+2$ & $T+8$ & line 5 \\
9 & M & $2\times(H+2)$ & $T+4$ & line 2 \\
10 & S & $(T+6)-1$ & $T+5$ & line 5 \\
11 & D & $(T+8)/2$ & $H+4$ & line 8 \\
12 & A & $(T+6)+1$ & $T+7$ & line 5 \\
\bottomrule
\end{tabular}
\end{minipage}\hfill
\begin{minipage}[t]{0.48\textwidth}
\centering
\parbox{\linewidth}{\centering \textbf{Offset} $c = 9$ (deletes one round; 13 moves)}\par\smallskip
\small
\setlength{\tabcolsep}{3.5pt}
\begin{tabular}{rclll}
\toprule
\# & Op & Calculation & Result & From \\
\midrule
1 & M & $4\times(3q+2)$ & $T+2$ &  \\
2 & S & $(H+5)-3$ & $H+2$ & kept \\
3 & D & $(8q+6)/2$ & $P+1$ & kept \\
4 & A & $T+1$ & $T+1$ & kept \\
5 & M & $6\times(2q+2)$ & $T+6$ &  \\
6 & S & $(T+6)-3$ & $T+3$ & line 5 \\
7 & D & $(T+6)/2$ & $H+3$ & line 5 \\
8 & A & $(T+3)+2$ & $T+5$ & line 6 \\
9 & M & $3\times(P+3)$ & $T+9$ & kept \\
10 & S & $(T+9)-1$ & $T+8$ & line 9 \\
11 & D & $(T+8)/2$ & $H+4$ & line 10 \\
12 & A & $(T+6)+1$ & $T+7$ & line 5 \\
13 & M & $2\times(H+2)$ & $T+4$ & line 2 \\
\bottomrule
\end{tabular}
\end{minipage}

\bigskip

\noindent
\begin{minipage}[t]{0.48\textwidth}
\centering
\parbox{\linewidth}{\centering \textbf{Offset} $c = 10$ (deletes one round; 14 moves)}\par\smallskip
\small
\setlength{\tabcolsep}{3.5pt}
\begin{tabular}{rclll}
\toprule
\# & Op & Calculation & Result & From \\
\midrule
1 & M & $4\times(3q+2)$ & $T+2$ &  \\
2 & S & $(H+5)-3$ & $H+2$ & kept \\
3 & D & $(8q+6)/2$ & $P+1$ & kept \\
4 & A & $T+1$ & $T+1$ & kept \\
5 & M & $6\times(2q+2)$ & $T+6$ &  \\
6 & S & $(T+6)-3$ & $T+3$ & line 5 \\
7 & D & $(T+6)/2$ & $H+3$ & line 5 \\
8 & A & $(T+3)+1$ & $T+4$ & line 6 \\
9 & M & $3\times(P+3)$ & $T+9$ & kept \\
10 & S & $(T+9)-1$ & $T+8$ & line 9 \\
11 & D & $(T+8)/2$ & $H+4$ & line 10 \\
12 & A & $(T+4)+1$ & $T+5$ & line 8 \\
13 & M & $4\times(3q+4)$ & $T+10$ &  \\
14 & S & $(T+8)-1$ & $T+7$ & line 10 \\
\bottomrule
\end{tabular}
\end{minipage}\hfill
\begin{minipage}[t]{0.48\textwidth}
\centering
\parbox{\linewidth}{\centering \textbf{Offset} $c = 11$ (deletes one round; 15 moves)}\par\smallskip
\small
\setlength{\tabcolsep}{3.5pt}
\begin{tabular}{rclll}
\toprule
\# & Op & Calculation & Result & From \\
\midrule
1 & M & $4\times(3q+2)$ & $T+2$ &  \\
2 & S & $(T+2)-1$ & $T+1$ & line 1 \\
3 & D & $(8q+6)/2$ & $P+1$ & kept \\
4 & A & $(T+2)+1$ & $T+3$ & line 1 \\
5 & M & $6\times(2q+2)$ & $T+6$ &  \\
6 & S & $(T+6)-2$ & $T+4$ & line 5 \\
7 & D & $(T+4)/2$ & $H+2$ & line 6 \\
8 & A & $(T+4)+1$ & $T+5$ & line 6 \\
9 & M & $3\times(P+3)$ & $T+9$ & kept \\
10 & S & $(T+9)-2$ & $T+7$ & line 9 \\
11 & D & $(T+6)/2$ & $H+3$ & line 5 \\
12 & A & $(T+9)+2$ & $T+11$ & line 9 \\
13 & M & $4\times(3q+4)$ & $T+10$ &  \\
14 & S & $(T+9)-1$ & $T+8$ & line 9 \\
15 & D & $(T+8)/2$ & $H+4$ & line 14 \\
\bottomrule
\end{tabular}
\end{minipage}

\bigskip

\begin{center}
\begin{minipage}[t]{0.48\textwidth}
\centering
\parbox{\linewidth}{\centering \textbf{Offset} $c = 15$ (deletes two rounds; 23 moves)}\par\smallskip
\small
\setlength{\tabcolsep}{3.5pt}
\begin{tabular}{rclll}
\toprule
\# & Op & Calculation & Result & From \\
\midrule
1 & M & $3\times(P+5)$ & $T+15$ & kept \\
2 & S & $(T+15)-1$ & $T+14$ & line 1 \\
3 & D & $(T+14)/2$ & $H+7$ & line 2 \\
4 & A & $(H+1)+3$ & $H+4$ & kept \\
5 & M & $6\times(2q+3)$ & $T+12$ &  \\
6 & S & $(T+12)-2$ & $T+10$ & line 5 \\
7 & D & $(T+12)/2$ & $H+6$ & line 5 \\
8 & A & $(H+1)+1$ & $H+2$ & kept \\
9 & M & $3\times(P+3)$ & $T+9$ & kept \\
10 & S & $(T+9)-3$ & $T+6$ & line 9 \\
11 & D & $(T+10)/2$ & $H+5$ & line 6 \\
12 & A & $(T+12)+1$ & $T+13$ & line 5 \\
13 & M & $2\times(H+4)$ & $T+8$ & line 4 \\
14 & S & $(T+12)-1$ & $T+11$ & line 5 \\
15 & D & $(T+6)/2$ & $H+3$ & line 10 \\
16 & A & $(T+6)+1$ & $T+7$ & line 10 \\
17 & M & $2\times(H+2)$ & $T+4$ & line 8 \\
18 & S & $(T+6)-1$ & $T+5$ & line 10 \\
19 & D & $(8q+8)/2$ & $P+2$ & kept \\
20 & A & $T+3$ & $T+3$ & kept \\
21 & M & $4\times(3q+2)$ & $T+2$ &  \\
22 & S & $(T+2)-1$ & $T+1$ & line 21 \\
23 & D & $(8q+6)/2$ & $P+1$ & kept$^{*}$ \\
\bottomrule
\end{tabular}
\par\smallskip
\parbox{\linewidth}{\footnotesize $^{*}$Here $8q+6 = 2P+2$. For $q \ge 3$ this number is kept; for $q = 2$ (that is, $P = 10$) it equals $H+7$, which line~3 of this patch makes first. Either way it is available.}
\end{minipage}
\end{center}

\section{Why the Patches Work for Every \texorpdfstring{$q$}{q}}\label{app:proof}

Every number in a patch is a linear expression in $q$, so each line of a patch table is an identity in $q$. We now check that every patch is legal for all $q$ at once, not only for the values tested by computer.

\begin{lemma}\label{lem:patch}
Let $P = 4q + 2$, and let either $c \in \{2, 4, 5, 6, 7, 8, 9, 10, 11, 15\}$ and $q \ge 2$, or $c = 1$ and $q \ge 3$. If $P$ has a perfect sequence from 2, then the tripling construction for $P$, with its final round (or final two rounds) replaced by the patch for $c$, is a perfect sequence for $3P + c$.
\end{lemma}

\begin{proof}
Let $d \in \{1, 2\}$ be the number of rounds the patch deletes, and let $R$ be the set of numbers those rounds made (the removed numbers).

The final round of the tripling construction is $j = P/2 - 1 = 2q$. Its multiplication gives $C_{2q} = 3(P - 2q) = 6q + 6 = H + 3$, so its subtraction and addition give $H + 2$ and $H + 4$, and its division gives $H - 2q = 4q + 3 = P + 1$. Since $2q \ge 4$, this is not one of the swapped rounds 1 and 2. In the same way, round $2q - 1$, which is at least 3, makes $H + 6$, $H + 5$, $H + 7$ and $H - (2q - 1) = P + 2$. Hence
\begin{align*}
  R &= \{P + 1,\ H + 2,\ H + 3,\ H + 4\} && (d = 1),\\
  R &= \{P + 1,\ P + 2,\ H + 2, \dots, H + 7\} && (d = 2).
\end{align*}
Tripling makes each number of $\{1, \dots, 3P\}$ exactly once, and the deleted rounds are the last ones, so no kept move uses a number from $R$. When the patch starts, the numbers available are therefore exactly
\[
  K = \{1, \dots, 3P\} \setminus R.
\]
A round begins on M, so the patch also begins on M.

\paragraph{}
Each table lists its operations in the order $\op{M}, \op{S}, \op{D}, \op{A}, \op{M}, \dots$. Expanding each calculation in terms of $q$ gives the stated result, for example $3 \times (2q + 2) = 6q + 6 = H + 3$ and $4 \times (3q + 2) = 12q + 8 = T + 2$. These are identities, so they hold for every $q$.

\paragraph{}
Every division is by 2, and every numerator is one of $8q + 6$, $8q + 8$ or $T + 2i$ with $2 \le i \le 7$. All of these are even for every $q$, since $T = 12q + 6$ is even.

\paragraph{}
The Result column of each patch lists the numbers of $R$ together with $T + 1, \dots, T + c$, each once. For $q \ge 1$ these are all different, since
\[
  P + 2 = 4q + 4 < 6q + 5 = H + 2 \qquad\text{and}\qquad H + 7 = 6q + 10 < 12q + 7 = T + 1.
\]
None of them is available before it is made, because the numbers of $R$ were deleted and $T + i > 3P$. So every move makes a new number of $\{1, \dots, 3P + c\}$.

\paragraph{}
An operand marked \emph{line $i$} is made earlier in the same patch. The other operands are of two kinds.
\begin{enumerate}
  \item \emph{Blank entries}, where both operands must lie in $\{1, \dots, P\}$. Apart from the constants $2, 3, 4$ and $6$, these operands are $2q + 2$, $2q + 3$, $3q + 2$, $3q + 4$ and $3q + 5$. The first four are at most $4q + 2 = P$ when $q \ge 2$. The last satisfies $3q + 5 \le P$ only when $q \ge 3$, and it appears only in line 9 of the patch for $c = 1$. This is why that patch needs $P \ge 14$.
  \item \emph{Kept entries}, where each operand must lie in $K$. Apart from the constants 1 to 4, the kept operands are
  \[
    H + 1, \quad H + 5, \quad H + 8, \quad P + 3, \quad P + 5, \quad 8q + 6, \quad 8q + 8, \quad T,
  \]
  where $8q + 6 = 2P + 2$ and $8q + 8 = 2P + 4$. Each is at most $3P$, so it is enough to check that none lies in $R$. Recall that $R \subseteq \{P + 1, P + 2\} \cup \{H + 2, \dots, H + 7\}$ and that $P + 2 < H$.
  \begin{itemize}
    \item $H + 1$, $H + 8$ and $T$ are outside both ranges.
    \item $H + 5$ is used only by the one-round patches $c = 4, \dots, 10$, where $H + 5 \notin R$.
    \item $P + 3$ and $P + 5$ are above $P + 2$, and below $H + 2$ because $4q + 7 < 6q + 5$ for $q \ge 2$.
    \item $2P + 4 = H + (2q + 5)$ and $2q + 5 \ge 9$, so it is above $H + 7$.
    \item $2P + 2 = H + (2q + 3)$. For $q \ge 3$ we have $2q + 3 \ge 9$, so it is above $H + 7$. For $q = 2$ it equals $H + 7$, which lies in $R$ only for the two-round patches $c = 1$ and $c = 15$. The patch for $c = 1$ is used only with $q \ge 3$, and in the patch for $c = 15$ line 3 makes $H + 7$ before line 23 uses it.
  \end{itemize}
\end{enumerate}

\paragraph{}
The patch has $4d + c$ moves and makes the $4d$ numbers of $R$ together with $c$ new ones. The whole sequence therefore has $(P - 1) + 2P + c = 3P + c - 1$ moves, and by the checks above it makes every number of $\{1, \dots, 3P + c\}$ exactly once. So it is perfect.
\end{proof}

Lemma~\ref{lem:patch} covers the patches for $c = 4$ and $c = 8$ used in Lemma~\ref{lem:everyP}, where $P' = 4r + 2$ with $r \ge 2$, and all the patches used in Theorem~\ref{thm:large}.

\section{Small Cases}\label{app:small}

All start from 2 and follow AMSD, reading down each column. Each was checked move by move.

\par\medskip
{\small\noindent
\makebox[\textwidth][c]{%
\makebox[0.2\textwidth][c]{%
\begin{tabular}[t]{@{}r@{\,}c@{\,}l@{\,}c@{\,}l@{}}
\multicolumn{5}{@{}c@{}}{$N = 11$} \\
$2$ & $+$ & $2$ & $=$ & $4$ \\
$2$ & $\times$ & $4$ & $=$ & $8$ \\
$8$ & $-$ & $2$ & $=$ & $6$ \\
$6$ & $/$ & $2$ & $=$ & $3$ \\
$4$ & $+$ & $3$ & $=$ & $7$ \\
$3$ & $\times$ & $3$ & $=$ & $9$ \\
$8$ & $-$ & $3$ & $=$ & $5$ \\
$2$ & $/$ & $2$ & $=$ & $1$ \\
$2$ & $+$ & $9$ & $=$ & $11$ \\
$2$ & $\times$ & $5$ & $=$ & $10$ \\
\end{tabular}}%
\makebox[0.2\textwidth][c]{%
\begin{tabular}[t]{@{}r@{\,}c@{\,}l@{\,}c@{\,}l@{}}
\multicolumn{5}{@{}c@{}}{$N = 12$} \\
$2$ & $+$ & $2$ & $=$ & $4$ \\
$2$ & $\times$ & $4$ & $=$ & $8$ \\
$8$ & $-$ & $2$ & $=$ & $6$ \\
$2$ & $/$ & $2$ & $=$ & $1$ \\
$2$ & $+$ & $8$ & $=$ & $10$ \\
$2$ & $\times$ & $6$ & $=$ & $12$ \\
$4$ & $-$ & $1$ & $=$ & $3$ \\
$10$ & $/$ & $2$ & $=$ & $5$ \\
$2$ & $+$ & $5$ & $=$ & $7$ \\
$3$ & $\times$ & $3$ & $=$ & $9$ \\
$12$ & $-$ & $1$ & $=$ & $11$ \\
\end{tabular}}%
\makebox[0.2\textwidth][c]{%
\begin{tabular}[t]{@{}r@{\,}c@{\,}l@{\,}c@{\,}l@{}}
\multicolumn{5}{@{}c@{}}{$N = 13$} \\
$2$ & $+$ & $2$ & $=$ & $4$ \\
$2$ & $\times$ & $4$ & $=$ & $8$ \\
$8$ & $-$ & $2$ & $=$ & $6$ \\
$6$ & $/$ & $2$ & $=$ & $3$ \\
$8$ & $+$ & $3$ & $=$ & $11$ \\
$2$ & $\times$ & $6$ & $=$ & $12$ \\
$12$ & $-$ & $2$ & $=$ & $10$ \\
$10$ & $/$ & $2$ & $=$ & $5$ \\
$2$ & $+$ & $11$ & $=$ & $13$ \\
$3$ & $\times$ & $3$ & $=$ & $9$ \\
$11$ & $-$ & $4$ & $=$ & $7$ \\
$2$ & $/$ & $2$ & $=$ & $1$ \\
\end{tabular}}%
\makebox[0.2\textwidth][c]{%
\begin{tabular}[t]{@{}r@{\,}c@{\,}l@{\,}c@{\,}l@{}}
\multicolumn{5}{@{}c@{}}{$N = 15$} \\
$2$ & $+$ & $2$ & $=$ & $4$ \\
$2$ & $\times$ & $4$ & $=$ & $8$ \\
$8$ & $-$ & $2$ & $=$ & $6$ \\
$2$ & $/$ & $2$ & $=$ & $1$ \\
$8$ & $+$ & $6$ & $=$ & $14$ \\
$2$ & $\times$ & $6$ & $=$ & $12$ \\
$14$ & $-$ & $1$ & $=$ & $13$ \\
$14$ & $/$ & $2$ & $=$ & $7$ \\
$2$ & $+$ & $1$ & $=$ & $3$ \\
$3$ & $\times$ & $3$ & $=$ & $9$ \\
$14$ & $-$ & $4$ & $=$ & $10$ \\
$10$ & $/$ & $2$ & $=$ & $5$ \\
$2$ & $+$ & $9$ & $=$ & $11$ \\
$3$ & $\times$ & $5$ & $=$ & $15$ \\
\end{tabular}}%
\makebox[0.2\textwidth][c]{%
\begin{tabular}[t]{@{}r@{\,}c@{\,}l@{\,}c@{\,}l@{}}
\multicolumn{5}{@{}c@{}}{$N = 16$} \\
$2$ & $+$ & $2$ & $=$ & $4$ \\
$4$ & $\times$ & $4$ & $=$ & $16$ \\
$16$ & $-$ & $2$ & $=$ & $14$ \\
$14$ & $/$ & $2$ & $=$ & $7$ \\
$2$ & $+$ & $4$ & $=$ & $6$ \\
$2$ & $\times$ & $4$ & $=$ & $8$ \\
$7$ & $-$ & $2$ & $=$ & $5$ \\
$2$ & $/$ & $2$ & $=$ & $1$ \\
$7$ & $+$ & $6$ & $=$ & $13$ \\
$2$ & $\times$ & $6$ & $=$ & $12$ \\
$16$ & $-$ & $7$ & $=$ & $9$ \\
$6$ & $/$ & $2$ & $=$ & $3$ \\
$2$ & $+$ & $9$ & $=$ & $11$ \\
$5$ & $\times$ & $3$ & $=$ & $15$ \\
$16$ & $-$ & $6$ & $=$ & $10$ \\
\end{tabular}}%
}
\par}\smallskip

\par\medskip
{\small\noindent
\makebox[\textwidth][c]{%
\makebox[0.2\textwidth][c]{%
\begin{tabular}[t]{@{}r@{\,}c@{\,}l@{\,}c@{\,}l@{}}
\multicolumn{5}{@{}c@{}}{$N = 17$} \\
$2$ & $+$ & $2$ & $=$ & $4$ \\
$4$ & $\times$ & $4$ & $=$ & $16$ \\
$16$ & $-$ & $2$ & $=$ & $14$ \\
$14$ & $/$ & $2$ & $=$ & $7$ \\
$4$ & $+$ & $7$ & $=$ & $11$ \\
$2$ & $\times$ & $4$ & $=$ & $8$ \\
$14$ & $-$ & $11$ & $=$ & $3$ \\
$2$ & $/$ & $2$ & $=$ & $1$ \\
$16$ & $+$ & $1$ & $=$ & $17$ \\
$3$ & $\times$ & $3$ & $=$ & $9$ \\
$16$ & $-$ & $1$ & $=$ & $15$ \\
$15$ & $/$ & $3$ & $=$ & $5$ \\
$4$ & $+$ & $8$ & $=$ & $12$ \\
$2$ & $\times$ & $5$ & $=$ & $10$ \\
$16$ & $-$ & $3$ & $=$ & $13$ \\
$12$ & $/$ & $2$ & $=$ & $6$ \\
\end{tabular}}%
\makebox[0.2\textwidth][c]{%
\begin{tabular}[t]{@{}r@{\,}c@{\,}l@{\,}c@{\,}l@{}}
\multicolumn{5}{@{}c@{}}{$N = 19$} \\
$2$ & $+$ & $2$ & $=$ & $4$ \\
$4$ & $\times$ & $4$ & $=$ & $16$ \\
$16$ & $-$ & $2$ & $=$ & $14$ \\
$2$ & $/$ & $2$ & $=$ & $1$ \\
$2$ & $+$ & $1$ & $=$ & $3$ \\
$3$ & $\times$ & $3$ & $=$ & $9$ \\
$16$ & $-$ & $9$ & $=$ & $7$ \\
$16$ & $/$ & $2$ & $=$ & $8$ \\
$16$ & $+$ & $1$ & $=$ & $17$ \\
$4$ & $\times$ & $3$ & $=$ & $12$ \\
$14$ & $-$ & $3$ & $=$ & $11$ \\
$12$ & $/$ & $2$ & $=$ & $6$ \\
$2$ & $+$ & $17$ & $=$ & $19$ \\
$2$ & $\times$ & $9$ & $=$ & $18$ \\
$16$ & $-$ & $1$ & $=$ & $15$ \\
$15$ & $/$ & $3$ & $=$ & $5$ \\
$2$ & $+$ & $11$ & $=$ & $13$ \\
$2$ & $\times$ & $5$ & $=$ & $10$ \\
\end{tabular}}%
\makebox[0.2\textwidth][c]{%
\begin{tabular}[t]{@{}r@{\,}c@{\,}l@{\,}c@{\,}l@{}}
\multicolumn{5}{@{}c@{}}{$N = 20$} \\
$2$ & $+$ & $2$ & $=$ & $4$ \\
$4$ & $\times$ & $4$ & $=$ & $16$ \\
$16$ & $-$ & $4$ & $=$ & $12$ \\
$2$ & $/$ & $2$ & $=$ & $1$ \\
$2$ & $+$ & $1$ & $=$ & $3$ \\
$3$ & $\times$ & $3$ & $=$ & $9$ \\
$12$ & $-$ & $1$ & $=$ & $11$ \\
$16$ & $/$ & $2$ & $=$ & $8$ \\
$2$ & $+$ & $11$ & $=$ & $13$ \\
$2$ & $\times$ & $3$ & $=$ & $6$ \\
$16$ & $-$ & $2$ & $=$ & $14$ \\
$14$ & $/$ & $2$ & $=$ & $7$ \\
$4$ & $+$ & $16$ & $=$ & $20$ \\
$2$ & $\times$ & $9$ & $=$ & $18$ \\
$20$ & $-$ & $3$ & $=$ & $17$ \\
$20$ & $/$ & $4$ & $=$ & $5$ \\
$2$ & $+$ & $17$ & $=$ & $19$ \\
$2$ & $\times$ & $5$ & $=$ & $10$ \\
$16$ & $-$ & $1$ & $=$ & $15$ \\
\end{tabular}}%
\makebox[0.2\textwidth][c]{%
\begin{tabular}[t]{@{}r@{\,}c@{\,}l@{\,}c@{\,}l@{}}
\multicolumn{5}{@{}c@{}}{$N = 21$} \\
$2$ & $+$ & $2$ & $=$ & $4$ \\
$4$ & $\times$ & $4$ & $=$ & $16$ \\
$16$ & $-$ & $2$ & $=$ & $14$ \\
$14$ & $/$ & $2$ & $=$ & $7$ \\
$2$ & $+$ & $4$ & $=$ & $6$ \\
$2$ & $\times$ & $4$ & $=$ & $8$ \\
$7$ & $-$ & $6$ & $=$ & $1$ \\
$6$ & $/$ & $2$ & $=$ & $3$ \\
$14$ & $+$ & $1$ & $=$ & $15$ \\
$6$ & $\times$ & $3$ & $=$ & $18$ \\
$16$ & $-$ & $4$ & $=$ & $12$ \\
$15$ & $/$ & $3$ & $=$ & $5$ \\
$7$ & $+$ & $6$ & $=$ & $13$ \\
$4$ & $\times$ & $5$ & $=$ & $20$ \\
$18$ & $-$ & $1$ & $=$ & $17$ \\
$18$ & $/$ & $2$ & $=$ & $9$ \\
$2$ & $+$ & $17$ & $=$ & $19$ \\
$7$ & $\times$ & $3$ & $=$ & $21$ \\
$16$ & $-$ & $5$ & $=$ & $11$ \\
$20$ & $/$ & $2$ & $=$ & $10$ \\
\end{tabular}}%
\makebox[0.2\textwidth][c]{%
\begin{tabular}[t]{@{}r@{\,}c@{\,}l@{\,}c@{\,}l@{}}
\multicolumn{5}{@{}c@{}}{$N = 23$} \\
$2$ & $+$ & $2$ & $=$ & $4$ \\
$4$ & $\times$ & $4$ & $=$ & $16$ \\
$16$ & $-$ & $2$ & $=$ & $14$ \\
$2$ & $/$ & $2$ & $=$ & $1$ \\
$2$ & $+$ & $1$ & $=$ & $3$ \\
$4$ & $\times$ & $3$ & $=$ & $12$ \\
$16$ & $-$ & $1$ & $=$ & $15$ \\
$14$ & $/$ & $2$ & $=$ & $7$ \\
$2$ & $+$ & $3$ & $=$ & $5$ \\
$3$ & $\times$ & $7$ & $=$ & $21$ \\
$16$ & $-$ & $5$ & $=$ & $11$ \\
$16$ & $/$ & $2$ & $=$ & $8$ \\
$14$ & $+$ & $8$ & $=$ & $22$ \\
$3$ & $\times$ & $3$ & $=$ & $9$ \\
$16$ & $-$ & $3$ & $=$ & $13$ \\
$12$ & $/$ & $2$ & $=$ & $6$ \\
$2$ & $+$ & $15$ & $=$ & $17$ \\
$4$ & $\times$ & $5$ & $=$ & $20$ \\
$21$ & $-$ & $2$ & $=$ & $19$ \\
$20$ & $/$ & $2$ & $=$ & $10$ \\
$2$ & $+$ & $21$ & $=$ & $23$ \\
$2$ & $\times$ & $9$ & $=$ & $18$ \\
\end{tabular}}%
}
\par}\smallskip

\par\medskip
{\small\noindent
\makebox[\textwidth][c]{%
\makebox[0.2\textwidth][c]{%
\begin{tabular}[t]{@{}r@{\,}c@{\,}l@{\,}c@{\,}l@{}}
\multicolumn{5}{@{}c@{}}{$N = 24$} \\
$2$ & $+$ & $2$ & $=$ & $4$ \\
$4$ & $\times$ & $4$ & $=$ & $16$ \\
$16$ & $-$ & $4$ & $=$ & $12$ \\
$12$ & $/$ & $2$ & $=$ & $6$ \\
$4$ & $+$ & $6$ & $=$ & $10$ \\
$2$ & $\times$ & $4$ & $=$ & $8$ \\
$16$ & $-$ & $2$ & $=$ & $14$ \\
$2$ & $/$ & $2$ & $=$ & $1$ \\
$4$ & $+$ & $1$ & $=$ & $5$ \\
$2$ & $\times$ & $10$ & $=$ & $20$ \\
$16$ & $-$ & $1$ & $=$ & $15$ \\
$12$ & $/$ & $4$ & $=$ & $3$ \\
$4$ & $+$ & $5$ & $=$ & $9$ \\
$2$ & $\times$ & $9$ & $=$ & $18$ \\
$20$ & $-$ & $3$ & $=$ & $17$ \\
$14$ & $/$ & $2$ & $=$ & $7$ \\
$2$ & $+$ & $20$ & $=$ & $22$ \\
$3$ & $\times$ & $7$ & $=$ & $21$ \\
$20$ & $-$ & $1$ & $=$ & $19$ \\
$22$ & $/$ & $2$ & $=$ & $11$ \\
$2$ & $+$ & $21$ & $=$ & $23$ \\
$2$ & $\times$ & $12$ & $=$ & $24$ \\
$16$ & $-$ & $3$ & $=$ & $13$ \\
\end{tabular}}%
\makebox[0.2\textwidth][c]{%
\begin{tabular}[t]{@{}r@{\,}c@{\,}l@{\,}c@{\,}l@{}}
\multicolumn{5}{@{}c@{}}{$N = 25$} \\
$2$ & $+$ & $2$ & $=$ & $4$ \\
$4$ & $\times$ & $4$ & $=$ & $16$ \\
$16$ & $-$ & $4$ & $=$ & $12$ \\
$12$ & $/$ & $4$ & $=$ & $3$ \\
$2$ & $+$ & $4$ & $=$ & $6$ \\
$3$ & $\times$ & $3$ & $=$ & $9$ \\
$9$ & $-$ & $4$ & $=$ & $5$ \\
$16$ & $/$ & $2$ & $=$ & $8$ \\
$16$ & $+$ & $8$ & $=$ & $24$ \\
$5$ & $\times$ & $5$ & $=$ & $25$ \\
$24$ & $-$ & $2$ & $=$ & $22$ \\
$2$ & $/$ & $2$ & $=$ & $1$ \\
$4$ & $+$ & $16$ & $=$ & $20$ \\
$3$ & $\times$ & $5$ & $=$ & $15$ \\
$16$ & $-$ & $3$ & $=$ & $13$ \\
$20$ & $/$ & $2$ & $=$ & $10$ \\
$3$ & $+$ & $20$ & $=$ & $23$ \\
$2$ & $\times$ & $9$ & $=$ & $18$ \\
$16$ & $-$ & $2$ & $=$ & $14$ \\
$14$ & $/$ & $2$ & $=$ & $7$ \\
$4$ & $+$ & $15$ & $=$ & $19$ \\
$3$ & $\times$ & $7$ & $=$ & $21$ \\
$24$ & $-$ & $7$ & $=$ & $17$ \\
$22$ & $/$ & $2$ & $=$ & $11$ \\
\end{tabular}}%
\makebox[0.2\textwidth][c]{%
\begin{tabular}[t]{@{}r@{\,}c@{\,}l@{\,}c@{\,}l@{}}
\multicolumn{5}{@{}c@{}}{$N = 27$} \\
$2$ & $+$ & $2$ & $=$ & $4$ \\
$2$ & $\times$ & $4$ & $=$ & $8$ \\
$8$ & $-$ & $2$ & $=$ & $6$ \\
$2$ & $/$ & $2$ & $=$ & $1$ \\
$2$ & $+$ & $8$ & $=$ & $10$ \\
$2$ & $\times$ & $6$ & $=$ & $12$ \\
$12$ & $-$ & $1$ & $=$ & $11$ \\
$10$ & $/$ & $2$ & $=$ & $5$ \\
$8$ & $+$ & $11$ & $=$ & $19$ \\
$2$ & $\times$ & $11$ & $=$ & $22$ \\
$19$ & $-$ & $1$ & $=$ & $18$ \\
$6$ & $/$ & $2$ & $=$ & $3$ \\
$4$ & $+$ & $12$ & $=$ & $16$ \\
$5$ & $\times$ & $3$ & $=$ & $15$ \\
$19$ & $-$ & $2$ & $=$ & $17$ \\
$18$ & $/$ & $2$ & $=$ & $9$ \\
$2$ & $+$ & $19$ & $=$ & $21$ \\
$2$ & $\times$ & $10$ & $=$ & $20$ \\
$19$ & $-$ & $5$ & $=$ & $14$ \\
$21$ & $/$ & $3$ & $=$ & $7$ \\
$8$ & $+$ & $19$ & $=$ & $27$ \\
$2$ & $\times$ & $12$ & $=$ & $24$ \\
$27$ & $-$ & $1$ & $=$ & $26$ \\
$26$ & $/$ & $2$ & $=$ & $13$ \\
$2$ & $+$ & $21$ & $=$ & $23$ \\
$5$ & $\times$ & $5$ & $=$ & $25$ \\
\end{tabular}}%
\makebox[0.2\textwidth][c]{%
\begin{tabular}[t]{@{}r@{\,}c@{\,}l@{\,}c@{\,}l@{}}
\multicolumn{5}{@{}c@{}}{$N = 28$} \\
$2$ & $+$ & $2$ & $=$ & $4$ \\
$4$ & $\times$ & $4$ & $=$ & $16$ \\
$16$ & $-$ & $2$ & $=$ & $14$ \\
$14$ & $/$ & $2$ & $=$ & $7$ \\
$4$ & $+$ & $4$ & $=$ & $8$ \\
$2$ & $\times$ & $14$ & $=$ & $28$ \\
$14$ & $-$ & $4$ & $=$ & $10$ \\
$2$ & $/$ & $2$ & $=$ & $1$ \\
$14$ & $+$ & $1$ & $=$ & $15$ \\
$2$ & $\times$ & $10$ & $=$ & $20$ \\
$28$ & $-$ & $10$ & $=$ & $18$ \\
$18$ & $/$ & $2$ & $=$ & $9$ \\
$2$ & $+$ & $10$ & $=$ & $12$ \\
$2$ & $\times$ & $12$ & $=$ & $24$ \\
$28$ & $-$ & $9$ & $=$ & $19$ \\
$12$ & $/$ & $4$ & $=$ & $3$ \\
$2$ & $+$ & $24$ & $=$ & $26$ \\
$9$ & $\times$ & $3$ & $=$ & $27$ \\
$24$ & $-$ & $2$ & $=$ & $22$ \\
$26$ & $/$ & $2$ & $=$ & $13$ \\
$2$ & $+$ & $3$ & $=$ & $5$ \\
$5$ & $\times$ & $5$ & $=$ & $25$ \\
$20$ & $-$ & $3$ & $=$ & $17$ \\
$18$ & $/$ & $3$ & $=$ & $6$ \\
$4$ & $+$ & $19$ & $=$ & $23$ \\
$7$ & $\times$ & $3$ & $=$ & $21$ \\
$16$ & $-$ & $5$ & $=$ & $11$ \\
\end{tabular}}%
\makebox[0.2\textwidth][c]{%
\begin{tabular}[t]{@{}r@{\,}c@{\,}l@{\,}c@{\,}l@{}}
\multicolumn{5}{@{}c@{}}{$N = 29$} \\
$2$ & $+$ & $2$ & $=$ & $4$ \\
$4$ & $\times$ & $4$ & $=$ & $16$ \\
$16$ & $-$ & $2$ & $=$ & $14$ \\
$16$ & $/$ & $2$ & $=$ & $8$ \\
$2$ & $+$ & $8$ & $=$ & $10$ \\
$2$ & $\times$ & $14$ & $=$ & $28$ \\
$28$ & $-$ & $4$ & $=$ & $24$ \\
$10$ & $/$ & $2$ & $=$ & $5$ \\
$10$ & $+$ & $5$ & $=$ & $15$ \\
$5$ & $\times$ & $5$ & $=$ & $25$ \\
$24$ & $-$ & $5$ & $=$ & $19$ \\
$24$ & $/$ & $8$ & $=$ & $3$ \\
$2$ & $+$ & $16$ & $=$ & $18$ \\
$2$ & $\times$ & $3$ & $=$ & $6$ \\
$25$ & $-$ & $8$ & $=$ & $17$ \\
$24$ & $/$ & $2$ & $=$ & $12$ \\
$8$ & $+$ & $5$ & $=$ & $13$ \\
$2$ & $\times$ & $13$ & $=$ & $26$ \\
$28$ & $-$ & $6$ & $=$ & $22$ \\
$22$ & $/$ & $2$ & $=$ & $11$ \\
$4$ & $+$ & $19$ & $=$ & $23$ \\
$3$ & $\times$ & $3$ & $=$ & $9$ \\
$28$ & $-$ & $8$ & $=$ & $20$ \\
$2$ & $/$ & $2$ & $=$ & $1$ \\
$4$ & $+$ & $25$ & $=$ & $29$ \\
$3$ & $\times$ & $9$ & $=$ & $27$ \\
$24$ & $-$ & $3$ & $=$ & $21$ \\
$14$ & $/$ & $2$ & $=$ & $7$ \\
\end{tabular}}%
}
\par}\smallskip

\par\medskip
{\small\noindent
\makebox[\textwidth][c]{%
\makebox[0.2\textwidth][c]{%
\begin{tabular}[t]{@{}r@{\,}c@{\,}l@{\,}c@{\,}l@{}}
\multicolumn{5}{@{}c@{}}{$N = 31$} \\
$2$ & $+$ & $2$ & $=$ & $4$ \\
$4$ & $\times$ & $4$ & $=$ & $16$ \\
$16$ & $-$ & $2$ & $=$ & $14$ \\
$16$ & $/$ & $2$ & $=$ & $8$ \\
$14$ & $+$ & $8$ & $=$ & $22$ \\
$2$ & $\times$ & $14$ & $=$ & $28$ \\
$16$ & $-$ & $4$ & $=$ & $12$ \\
$12$ & $/$ & $4$ & $=$ & $3$ \\
$2$ & $+$ & $3$ & $=$ & $5$ \\
$3$ & $\times$ & $3$ & $=$ & $9$ \\
$16$ & $-$ & $3$ & $=$ & $13$ \\
$14$ & $/$ & $2$ & $=$ & $7$ \\
$2$ & $+$ & $16$ & $=$ & $18$ \\
$3$ & $\times$ & $7$ & $=$ & $21$ \\
$28$ & $-$ & $3$ & $=$ & $25$ \\
$22$ & $/$ & $2$ & $=$ & $11$ \\
$4$ & $+$ & $13$ & $=$ & $17$ \\
$2$ & $\times$ & $12$ & $=$ & $24$ \\
$28$ & $-$ & $5$ & $=$ & $23$ \\
$12$ & $/$ & $2$ & $=$ & $6$ \\
$2$ & $+$ & $24$ & $=$ & $26$ \\
$5$ & $\times$ & $6$ & $=$ & $30$ \\
$22$ & $-$ & $3$ & $=$ & $19$ \\
$30$ & $/$ & $3$ & $=$ & $10$ \\
$14$ & $+$ & $17$ & $=$ & $31$ \\
$3$ & $\times$ & $9$ & $=$ & $27$ \\
$22$ & $-$ & $7$ & $=$ & $15$ \\
$2$ & $/$ & $2$ & $=$ & $1$ \\
$2$ & $+$ & $27$ & $=$ & $29$ \\
$2$ & $\times$ & $10$ & $=$ & $20$ \\
\end{tabular}}%
\makebox[0.2\textwidth][c]{%
\begin{tabular}[t]{@{}r@{\,}c@{\,}l@{\,}c@{\,}l@{}}
\multicolumn{5}{@{}c@{}}{$N = 33$} \\
$2$ & $+$ & $2$ & $=$ & $4$ \\
$4$ & $\times$ & $4$ & $=$ & $16$ \\
$16$ & $-$ & $2$ & $=$ & $14$ \\
$14$ & $/$ & $2$ & $=$ & $7$ \\
$4$ & $+$ & $7$ & $=$ & $11$ \\
$2$ & $\times$ & $16$ & $=$ & $32$ \\
$16$ & $-$ & $7$ & $=$ & $9$ \\
$2$ & $/$ & $2$ & $=$ & $1$ \\
$32$ & $+$ & $1$ & $=$ & $33$ \\
$2$ & $\times$ & $14$ & $=$ & $28$ \\
$32$ & $-$ & $9$ & $=$ & $23$ \\
$16$ & $/$ & $2$ & $=$ & $8$ \\
$2$ & $+$ & $11$ & $=$ & $13$ \\
$2$ & $\times$ & $9$ & $=$ & $18$ \\
$33$ & $-$ & $4$ & $=$ & $29$ \\
$33$ & $/$ & $11$ & $=$ & $3$ \\
$4$ & $+$ & $13$ & $=$ & $17$ \\
$2$ & $\times$ & $3$ & $=$ & $6$ \\
$32$ & $-$ & $8$ & $=$ & $24$ \\
$24$ & $/$ & $2$ & $=$ & $12$ \\
$2$ & $+$ & $17$ & $=$ & $19$ \\
$2$ & $\times$ & $13$ & $=$ & $26$ \\
$32$ & $-$ & $2$ & $=$ & $30$ \\
$30$ & $/$ & $6$ & $=$ & $5$ \\
$2$ & $+$ & $29$ & $=$ & $31$ \\
$5$ & $\times$ & $5$ & $=$ & $25$ \\
$33$ & $-$ & $11$ & $=$ & $22$ \\
$30$ & $/$ & $3$ & $=$ & $10$ \\
$2$ & $+$ & $19$ & $=$ & $21$ \\
$9$ & $\times$ & $3$ & $=$ & $27$ \\
$32$ & $-$ & $12$ & $=$ & $20$ \\
$30$ & $/$ & $2$ & $=$ & $15$ \\
\end{tabular}}%
}
\par}\smallskip


\begin{thebibliography}{9}

\bibitem{DLS}
P. Downey, B. Leong and R. Sethi, Computing sequences with addition chains, \emph{SIAM J. Comput.} \textbf{10}(3) (1981), 638--646.

\bibitem{BN}
H. M. Bahig and D. I. Nassr, Generating a shortest $B$-chain using multi-GPUs, \emph{Inf. Sci. Lett.} \textbf{11}(3) (2022), 745--750.

\bibitem{Knuth}
D. E. Knuth, \emph{The Art of Computer Programming, Vol.~2: Seminumerical Algorithms}, 3rd ed., Addison-Wesley, 1997, $\S$4.6.3.

\bibitem{OEIS}
OEIS Foundation, Sequence A005245 (integer complexity), \url{https://oeis.org/A005245}.

\end{thebibliography}
\end{document}